\documentclass[11pt,english]{amsart} 

\usepackage[T1]{fontenc}
\usepackage{xcolor}
\usepackage{CJKutf8}
\newcommand{\zh}[1]{\begin{CJK}{UTF8}{gbsn}#1\end{CJK}}
\usepackage{lmodern}
\usepackage{babel}
\usepackage{amssymb, latexsym, amsmath, amsfonts, hyperref, enumitem, fancyhdr}
\usepackage{amsmath,amscd}
\usepackage{cleveref}
\usepackage{mathrsfs}
\usepackage{mathtools}
\usepackage{a4wide}

\usepackage[all]{xy}

\usepackage[english=usenglishmax]{hyphsubst}
\usepackage{microtype}

	\title[Maximal ergodic inequalities on noncommutative $L_p$-spaces]{Maximal ergodic inequalities for some positive operators on noncommutative $L_p$-spaces}
	
	

\author{Guixiang Hong}
\author{Samya Kumar Ray}
\author{Simeng Wang}

\address{Guixiang Hong: Institute for Advanced Study in Mathematics, Harbin Institute of Technology, Harbin 150001, China.}
\email{guixiang.hong@whu.edu.cn}
\address{Samya Kumar Ray: Stat-Math Unit, Indian Statistical Institute, 203, B. T. Road, Kolkata 700108, India}
\email{samyaray7777@gmail.com}

\address{Institute for Advanced Study in Mathematics, Harbin Institute of Technology, Harbin 150001, China.}
\email{simeng.wang@hit.edu.cn}

\theoremstyle{plain}
\newtheorem{thm}{Theorem}[section]
\newtheorem{cor}[thm]{Corollary}
\newtheorem{lem}[thm]{Lemma}
\newtheorem{Qu}[thm]{Question}
\newtheorem{prop}[thm]{Proposition}

\theoremstyle{definition}
\newtheorem{defn}[thm]{Definition}
\newtheorem{exa}[thm]{Example}
\newtheorem{rem}[thm]{Remark}

\begin{document}

\begin{abstract}
In this paper, we push forward the conjecture on a noncommutative version of the maximal ergodic theorem for positive contractions due to Ackoglu. That is, we establish the one-sided maximal ergodic inequalities for a large subclass of positive operators on noncommutative $L_p$-spaces for a fixed $1<p<\infty$, which particularly applies to positive isometries, the convex hull of positive Lamperti contractions and also the power bounded doubly Lamperti operators; moreover, it is known that this subclass recovers all positive contractions on the classical Lebesgue spaces $L_p([0,1])$. As an unexpected consequence, we deduce a completely bounded version of Ackoglu's theorem by making use of Ackoglu's original dilation. We also observe that the concrete examples of positive contractions without Akcoglu's dilation,  which were constructed by Junge-Le Merdy \cite{jungelemerdy07dilation}, still satisfy the maximal ergodic inequalities. Together with the class of power bounded positive invertible operators considered by Hong-Liao-Wang \cite{hongliaowang15erg}, we thereby verify the noncommutative Ackoglu ergodic theorem for all the positive operators that have naturally appeared in the literature up to the moment of writing. Based on the noncommutative Calder\'on transference principle established in \cite{hongliaowang15erg}, the general pattern of the demonstration follows from the classical one due to Kan, but several new ideas are absolutely necessary in the noncommutative setting. Let us just mention three of them: the argument for the maximal ergodic theorem for positive isometries is highly nontrivial compared to the classical case; the novel \emph{simultaneous} dilation property is indispensable to get the ergodic theorem for the convex hull of positive Lamperti contractions; numerous adjustments are truly needed  in this new setting to conclude a desired structural theorem for positive doubly Lamperti operators since the orthogonal relations of operator algebras are completely different from those in classical measure theory.




\end{abstract}

\maketitle

\tableofcontents

\section{Introduction}\label{sec:intro}

	In classical ergodic theory, one of the earliest pointwise ergodic  theorems was obtained by Birkhoff \cite{birkhoff31erg} in 1931. In many situations, it is well-known that establishing a maximal ergodic inequality  is enough to obtain a pointwise ergodic theorem. For example, the Birkhoff ergodic theorem can be derived from a weak $(1,1)$ type estimate of the maximal operator corresponding to the time averages, which was obtained by Wiener \cite{weiner39ergodic}. Dunford and Schwartz \cite{dunfordschwartz56erg} greatly generalized the previous situation; they established the strong $(p,p)$  maximal inequalities for all $1<p<\infty$ for time averages of positive $L_1$-$L_\infty$ contractions. However, the most general result in this direction was obtained by Akcoglu \cite{akcoglu75erg}, who established a maximal ergodic inequality for general positive contractions on $L_p$-spaces for a fixed $1<p<\infty$. The proof is based on an ingenious dilation theorem (see also \cite{akcoglusucheston77dilation}) which reduces the problem to the case of positive isometries, and the latter was already studied by Ionescu Tulcea \cite{ionescutulcea64ergisometry}. Akcoglu's dilation theorem has found numerous applications in various directions; let us mention (among others) Peller's work on Matsaev's conjecture for contractions on $L_p$-spaces \cite{peller76vn,peller76poly,peller81vn,peller85polyschatten}, Coifman-Rochberg-Weiss' approach to Stein's Littlewood-Paley theory \cite{coifmanrochbergweiss78transference}, $g$-function type estimates on compact Riemannian manifolds by Coifman-Weiss \cite{coifmanweiss76transference}, as well as functional calculus of Ritt and sectorial operators (see \cite{arhancetlemerdy14dilationritt,lemerdy14hinftyritt,lemerdy99hinfty} and references therein). On the other hand, we would like to remark that the Lamperti contractions consist of a typical class of general $L_p$-contractions. Moreover, Kan \cite{kan78erglamperti} established a maximal ergodic inequality for power bounded Lamperti operators whose adjoints are also Lamperti. Many more results for positive operators and Lamperti operators in the context of ergodic theory were studied further by various authors. We refer to \cite{jonesolsenwierdl92subseq,jonesolsen92subseq,sato87erghilbert,lemerdyxu12max,lemerdyxu12qvariation} and references therein for interested readers.

Motivated by quantum physics, noncommutative mathematics have advanced in a rapid speed. The connection between ergodic theory and von Neumann algebras is intimate and goes back to the earlier development of the theory of rings of operators. However, the study of pointwise ergodic theorems only took off with the pioneering work of Lance \cite{lance76erg}. The topic was then stupendously studied in a series of works due to Conze, Dang-Ngoc \cite{conzedangngoc78erg}, K{\"u}mmerer \cite{kummerer78erg}, Yeadon \cite{yeadon77max} and others. However, the maximal inequalities and pointwise ergodic theorems in $L_p$-spaces remained out of reach for many years until the path-breaking work of Junge and Xu \cite{jungexu07erg}. In \cite{jungexu07erg}, the authors established a noncommutative analogue of Dunford-Schwartz maximal ergodic theorem. This breakthrough motivated further research to develop various noncommutative ergodic theorems. We refer to \cite{bekjan08ergpositive,anantharaman06freeerg,hu08freeerg,hongsun16,hongliaowang15erg} and references therein. Notice that the general positive contractions considered by Akcoglu do not fall into the category of Junge-Xu \cite{jungexu07erg}. In the noncommutative setting, there are very few results for operators beyond $L_1$-$L_\infty$ contractions except some isolated cases studied in \cite{hongliaowang15erg}. In particular, the following noncommutative analogue of Akcoglu's maximal ergodic inequalities, which is more general than Junge-Xu's results \cite{jungexu07erg}, remains open. We refer to Section \ref{PRE} for a precise presentation of the notation appearing here and below. 
\begin{Qu}\label{QUE} 
	Let $\mathcal M$ be a von Neumann algebra equipped with a normal faithful semifinite trace $\tau.$ Let $1<p<\infty$ and $T:L_p(\mathcal M )\to L_p(\mathcal M )$ be a positive contraction. Does  there exist  a positive constant $C$, such that \[\Big\|\Big(\frac{1}{ n+1}\sum_{k=0}^nT^k x\Big)_{n\geq 0}\Big\|_{L_p(\mathcal M;\ell_\infty)}\leq C \|x\|_p\] for all $x\in L_p(\mathcal M) $?
\end{Qu}
We refer to Subsection \ref{sub2.2} for the notion of the noncommutative vector-valued $L_p$-space $L_p(\mathcal{M};\ell_\infty).$ In this article, we answer Question \ref{QUE} for a large class of positive operators  which do not fall into  the category of aforementioned works. Indeed, this class recovers all positive contractions concerned in Question \ref{QUE} if $\mathcal M$ is the classical space $L_\infty ([0,1])$; moreover, we actually verify the noncommutative Ackoglu ergodic theorem for all the positive operators that have naturally appeared in the literature until the moment of writing. To introduce our main results we set some notation  and definitions.
\begin{defn}Let $1\leq p<\infty.$ Let $T:L_p(\mathcal{M},\tau )\to L_p(\mathcal{M},\tau )$ be a bounded linear map. We say that $T$ is a 
	 \emph{Lamperti} operator (or say that $T$ \emph{separates supports})  if for any two $\tau$-finite projections $e,f\in\mathcal{M}$ with $ef=0,$ we have \[(Te)^*Tf=Te(Tf)^*=0.\]
\end{defn}
By standard approximation arguments, it is easy to observe that the above definition of Lamperti operators agrees with the known one  in the commutative setting (also called ``separation-preserving operators" or ``disjoint operators" in some references); the study of the latter goes back to Banach \cite[Section XI.5]{banach32book}, and have subsequently been considered in various works (see e.g. \cite{asmarberksongillespie91transf,fendler97dilation,fendler98dilation, kan78erglamperti,peller76vn,peller81vn,peller85polyschatten}). We refer the readers to Section \ref{SLAM} for related properties of Lamperti operators in the noncommutative setting.

The following is one of our main results. Throughout the paper, we will denote by $C_p$ a fixed distinguished constant depending only on $p$, which is given by the best constant of Junge-Xu's maximal ergodic inequality \cite[Theorem 0.1]{jungexu07erg}.
\begin{thm}\label{main}
	Let $1<p<\infty.$ Assume that $T:L_p(\mathcal{M})\to L_p(\mathcal{M})$ belongs to the family
	\begin{equation}\label{eq:acssplus}
	\overline{\operatorname{conv}}^{\, sot} \{S:L_p(\mathcal{M} )\to L_p(\mathcal{M} ) \text{ positive  Lamperti contractions}  \}, 
	\end{equation}
	that is, the closed convex hull of all positive Lamperti contractions on $L_p(\mathcal{M})$ with respect to the strong operator topology. Then   \[\Big\|\Big(\frac{1}{ n+1}\sum_{k=0}^nT^k x\Big)_{n\geq 0}\Big\|_{L_p(\mathcal M;\ell_\infty)}\leq C_p \|x\|_p\]
	for all $x\in L_p(\mathcal M).$	\end{thm}
It is worth noticing that the class introduced in \eqref{eq:acssplus} is quite large in the classical setting. Indeed, together with \cite[Theorem 2]{grzp90approxpositive} and \cite{facklergluck19dilation}, we know that for $\mathcal M = L_\infty ([0,1])$ equipped with the Lebesgue measure, we have
\begin{align*}
&\{S:L_p([0,1])\to L_p([0,1] ) \text{ positive contractions}  \} \\
=\ & \overline{\operatorname{conv}}^{\, sot} \{S:L_p([0,1])\to L_p([0,1] ) \text{ positive   Lamperti contractions}  \},
\end{align*} 
which does recover the classical Akcoglu's ergodic theorem on $L_p([0,1] )$. Moreover, this method also helps to establish a completely bounded version of Ackoglu's ergodic theorem, which is a surprise, since as far as the authors know there is no existing theory to be applied or easy proof for this operator-valued analogue. See Corollary \ref{cbackoglu} and its proof.

As mentioned earlier, Akcoglu's arguments for ergodic theorem essentially rely on the study of dilations of positive contractions. In spite of various works on dilations on von Neumann algebras (see \cite{kummerer85markov,haagerupmusat11factorization,
jungericardshlyakhtenko2014noncommutative,arhancet20dilationlcg,arhancet19dilationvn, raymatsaev2020} and references therein),  Junge and Le Merdy showed in their remarkable paper \cite{jungelemerdy07dilation} that there is no `reasonable' analogue of Akcoglu's dilation theorem on noncommutative $L_p$-spaces. This becomes a serious difficulty in establishing a noncommutative analogue of Akcoglu's ergodic theorem. Our proof of the above theorem is based on the study of structural properties and simultaneous dilations of convex combinations of Lamperti operators as in \eqref{eq:acssplus}. This route is different from that of Akcoglu's original one. Let us mention some of the key steps and new ingredients in the proof, which might be of independent interest.
\begin{enumerate}
	\item[(i)]\textit{Noncommutative ergodic theorem for positive isometries} (Theorem \ref{DEEISO}): Following the classical case, the first natural step would be to establish a maximal ergodic inequality for positive isometries (see e.g. \cite{kan78erglamperti,ionescutulcea64ergisometry}). In this paper we give an analogue of this result in the noncommutative setting. However, the classical approach depending on `points' does not work in the noncommutative setting. The key ingredients are to extend positive Lamperti contractions and positive isometries on $L_p(\mathcal M)$ to the vector-valued space $L_p(\mathcal M;\ell_\infty)$ again as contractions and isometries respectively (see Lemma \ref{lem:extension lamperti} and Proposition \ref{prop:extension} respectively). These facts seem to be non-obvious. For example, restricted to  $2$-positive maps, Lemma \ref{lem:extension lamperti} follows from \cite[Proposition 7.3]{haajuxu10}. But extension of a positive contraction on $L_p(\mathcal M)$ to a `contraction' on $L_p(\mathcal M;\ell_\infty)$ is not guaranteed by \cite[Proposition 7.3]{haajuxu10}. However, we prove Lemma \ref{lem:extension lamperti} for any positive Lamperti contraction, by developing an approach completely independent of \cite[Proposition 7.3]{haajuxu10}. Then based on the transference techniques recently developed in \cite{hongliaowang15erg} combined with \cite[Theorem 0.1]{jungexu07erg}, we may obtain the desired maximal inequalities.  
	
	\item[(ii)]\textit{Structural theorems for Lamperti operators} (Theorem \ref{charac}, Theorem \ref{scharac}):
	In the classical setting, Peller \cite{peller76poly} and Kan \cite{kan78erglamperti} obtained a dilation theorem for Lamperti contractions. Their constructions are different from Akcoglu's and rely on structural descriptions of Lamperti operators. In the noncommutative setting, we first prove a similar characterization for Lamperti operators by using techniques from \cite{yeadon81isom}. Also, it is natural to consider the completely Lamperti operators in the noncommutative setting, and in this part we also prove a characterization theorem for these operators. This completes the second step for the proof of Theorem \ref{main}.  
	\item[(iii)]\textit{Dilation theorem for the convex hull of Lamperti contractions} (Theorem \ref{GDI}): In order to establish ergodic theorems for a large class \emph{beyond}   Lamperti contractions, we first prove a \emph{simultaneous} dilation theorem for tuples of   Lamperti contractions, which had not been observed by Peller and Kan and is a stronger version of their dilation theorem. The final step towards proving Theorem \ref{main} is to deploy tools from \cite{facklergluck19dilation} to obtain an $N$-dilation theorem for the convex hull of   Lamperti contractions for all $N\in\mathbb N$.   Our approach also establishes validity of \emph{noncommutative Matsaev's conjecture} for the strong closure of the closed convex hull of Lamperti contractions for $1<p\neq 2<\infty$  whenever the underlying von Neumann algebra has QWEP (see Corollary \ref{cor:matsaev} for details).

	It is worth mentioning that the dilatable contractions studied prior to our work are mostly those acting on the von Neumann algebra itself, except `loose dilation' results in \cite{arhancetlemerdy14dilationritt,arhancetfacklerlemerdy17dilationritt}. Also, our result might have some applications along the line of \cite{asmarberksongillespie91transf,coifmanrochbergweiss78transference,fendler97dilation,jonesolsenwierdl92subseq}. We leave this research direction open. 
\end{enumerate}

Note that Theroem \ref{main} only applies to  \emph{contractive} operators. As the classical case, the study for non-contractive power bounded operators requires additional efforts. In the following we  also establish a general ergodic theorem for power bounded Lamperti operators as soon as their adjoints are also Lamperti (usually called \emph{doubly Lamperti} operators), which is the other main result of the paper.
\begin{thm}\label{main1}Let $1<p<\infty$, $1/p+1/{p^\prime}=1$ and let $\mathcal M$ be a finite von Neumann algebra. Assume that $T:L_p(\mathcal M)\to L_p(\mathcal M)$ is a positive Lamperti operator such that the adjoint operator $T^*:L_{p^\prime}(\mathcal M)\to L_{p^\prime}(\mathcal M)$ is also Lamperti and $\sup_{n\geq 1}\|T^n\|_{L_p(\mathcal M)\to L_p(\mathcal M)}=K<\infty$. Then  \[\Big\|\Big(\frac{1}{ n+1}\sum_{k=0}^nT^k x\Big)_{n\geq 0}\Big\|_{L_p(\mathcal M;\ell_\infty)}\leq KC_p\|x\|_p\] for all $x\in L_p(\mathcal M).$
\end{thm}
The above theorem is the noncommutative analogue of a classical result of Kan \cite{kan78erglamperti}. It essentially relies on a structural theorem for positive doubly completely Lamperti operators (Theorem \ref{CHARACDL}), which reduces the problem to the setting of Theorem \ref{main}.
To prove this structural result, we follow the path of Kan. However, since the structures and orthogonal relations of von Neumann subalgebras are  completely different from those in classical measure theory, there appear lots of  new technical difficulties and numerous adjustments are truly needed in this new setting. Also, due to these technical reasons, we restrict our study to the case of finite von Neumann algebras only. The reader is strongly recommended to discover these non-trivialities  by adapting Kan's approach by himself/herself. 
\bigskip

Moreover, we observe that the maximal ergodic inequalities also hold for several other classes of operators outside the scope of Theorem \ref{main} or Theorem \ref{main1}. 

(i) \textit{Positive invertible operators which are not Lamperti} (Example \ref{ex:exa}):
Kan \cite{kan78erglamperti} discussed various examples of Lamperti operators. He showed that any positive invertible operator with positive inverse is Lamperti in the \emph{classical} setting. As a consequence, he reproved that any power bounded positive operator with positive inverse admits a maximal ergodic inequality; this generalized the ergodic theorem of de la Torre \cite{delatorre76erg}. A noncommutative analogue of this theorem, in a much general form, was achieved in \cite{hongliaowang15erg} (see Theorem \ref{inve}).  

However, in this article we provide examples of positive invertible operators on \emph{noncommutative} $L_p$-spaces with positive inverses which are \emph{not} even Lamperti. Therefore, Kan's
method does not immediately deduce de la Torre's ergodic theorem \cite{delatorre76erg} in the noncommutative setting.  Nevertheless, these examples fall into the category of the aforementioned result of \cite{hongliaowang15erg}, and hence satisfy the maximal ergodic theorem. We would like to remark that Kan's aforementioned examples of Lamperti operators play an important role in many other papers such as \cite{berksongillespie97meanbdd,jonesolsen92subseq,jonesolsenwierdl92subseq} and references therein. Kan \cite{kan78erglamperti} also showed that any positive invertible operator on a finite dimensional (commutative) $L_p$-space with $\sup_{n\in\mathbb Z}\|T^n\|_{L_p\to L_p}<\infty$ is Lamperti. Our example shows that this is again \emph{not} true in the noncommutative setting. All these phenomena seem to be new.

(ii) \textit{Junge-Le Merdy's non-dilatable example}: As mentioned earlier, there exist concrete examples of completely positive complete contractions which fail to admit a noncommutative analogue of Akcoglu's dilation, constructed by Junge and Le Merdy \cite{jungelemerdy07dilation}. In this paper we show that these operators still satisfy a maximal ergodic inequality. In particular we establish the following fact.
\begin{prop}\label{atm}
	Let $1<p\neq 2<\infty.$ Then for all $k\in\mathbb N$ large enough, there exists a completely positive complete contraction $T:S_p^k\to S_p^k$ such that \[\Big\|\Big(\frac{1}{ n+1}\sum_{k=0}^nT^k x\Big)_{n\geq 0}\Big\|_{L_p(\mathcal M;\ell_\infty)}\leq (C_p +1) \|x\|_p,\quad x\in L_p(\mathcal M),\]
	but $T$ does not have a dilation (in the sense of Definition \ref{dide1}).
\end{prop}
The proof is short and elementary, and relies surprisingly on {Akcoglu's ergodic theorem \cite{akcoglu75erg}} in the \emph{classical} setting.
The above theorem illustrates again that the noncommutative situation is significantly different  from the classical one. 

\bigskip

We end our introduction by briefly mentioning the organization of the paper. In Section \ref{PRE} we recall the necessary background including all the requisite definitions. In Section \ref{SLAM}, we prove the characterization theorems for Lamperti and completely Lamperti operators. In Section \ref{DILATION}, we prove the dilation theorem for the convex hull of Lamperti contractions, and establish the validity of noncommutative Matsaev's conjecture for this class of contractions. In Section \ref{isodee}, we prove the maximal ergodic inequalities for positive isometries and then deduce Theorem \ref{main} by applying the dilation theorem in Section \ref{DILATION}. Section \ref{lamp} is devoted to the proof of Theorem \ref{main1}, which involves some additional properties of Lamperti operators as well as an useful characterization theorem for doubly Lamperti operators. In Section \ref{exa}, we consider noncommutative ergodic theorems for various interesting operators which are out of the scope of Theorem \ref{main} and Theorem \ref{main1}. 

\medskip

After we finished the preliminary version of this paper, we learned that some partial results in Section \ref{SLAM} were also obtained independently in \cite{lemerdyzadeh19l1,lemerdyzadeh20isom} at the same time; a related study was also given in \cite{huangsukochevzanin20submaj}. However, both the main results and the arguments of this paper are  quite different and independent, which cannot be recovered from their works.

The main results of this paper were announced in \cite{hongraywang2020}.

\section{Preliminaries}\label{PRE}
\subsection{Noncommutaive $L_p$-spaces}
For any closed densely defined linear map  $T$ on a Banach space, we denote by $\ker T$ and $\operatorname{ran}\, T$ the kernel and range of $T$ respectively. Let $\mathcal M$ be a von Neumann algebra equipped with a normal semifinite faithful trace $\tau_\mathcal{M}$, which acts on a Hilbert space $\mathcal H.$ We also simply denote the trace $\tau_\mathcal{M}$ by $\tau$ if no confusion will occur.  Unless specified, we always work with von Neumann algebras of this kind. The unit in $\mathcal M$ is denoted by $1_{\mathcal M}$ or simply by $1$ and the extended positive cone of $\mathcal M$ is denoted by $\widehat{\mathcal M}_+$. Let  $L_0(\mathcal M)$  be the $*$-algebra of all closed densely defined  operators on $\mathcal H$  measurable with respect to $(\mathcal M ,\tau )$. For a subspace $A\subseteq L_0(\mathcal M)$, we denote by $A_+$ the cone of positive elements in $A$, and by  $\mathcal{Z}(A)$ the center of $A$ if $A$ is a subalgebra. The trace $\tau$ can be extended to $L_0(\mathcal M)_{+}$ and $\widehat{\mathcal M}_+$. 
A sequence $(x_n)_{n\geq 1}\subseteq L_0(\mathcal M)$ is said to \textit{converge in measure} to $x\in L_0(\mathcal M)$ if \[\forall\,\varepsilon>0 , \quad  \lim_{n\to\infty}\tau(e_\varepsilon^{\perp}(|x_n-x|))=0,\] where $e_\varepsilon^{\perp}(y)\coloneqq\chi_{(\varepsilon,\infty)}(y)$ for any $y\in L_0(\mathcal M)_+$ and $\chi$ denotes the usual characteristic function. 
We denote by $s(x)$ the support of  $x$ for a positive element $x\in L_0(\mathcal M)_{+}$. For any projection $e\in\mathcal M$ we denote $e^{\perp}=1-e.$

Let $\mathcal{S}(\mathcal{M})$ be the linear span of all positive elements in $\mathcal M$ such that $\tau(s(x))<\infty$. Let $\mathcal P(\mathcal M)$ denote the set of all projections in $\mathcal M.$ A projection $e\in \mathcal M$ is said to be $\tau$-finite if $e\in \mathcal{S}(\mathcal{M})$. For $1\leq p<\infty,$ we define the noncommutative $L_p$-space $L_p(\mathcal{M},\tau)$ to be the completion of $\mathcal{S}({\mathcal{M}})$ with respect to the norm \[\|x\|_{L_p(\mathcal{M})}\coloneqq\tau(|x|^p)^{\frac{1}{p}}, \quad \text{where } |x|=(x^*x)^{\frac{1}{2}}.\]
The Banach lattice structure of $L_p(\mathcal{M},\tau)$ does not depend on the choice of $\tau$ and we often simply denote the space by $L_p (\mathcal M)$ if no ambiguity will occur. We set $L_{\infty}(\mathcal{M})=\mathcal{M}$. It is well-known that $L_p (\mathcal M)$ can be viewed as a subspace of $L_0 (\mathcal M)$. For any $\sigma$-finite measure space $(\Omega,\mu),$ we have a natural identification for $L_p(L_\infty(\Omega)\overline{\otimes} \mathcal M)$ as the Bochner space $L_p(\Omega;L_p( \mathcal M))$ for $1\leq p<\infty.$

If $\mathcal{M}=B(\mathcal{H})$ for a Hilbert space $\mathcal{H}$ and if $\tau$ is the usual trace $Tr$ on it, then the corresponding noncommutative $L_p$-spaces are usually called Schatten-$p$ classes and denoted by $S_p(\mathcal{H})$ for $1\leq p<\infty$. When $\mathcal H$ is $\ell_2^n$ or $\ell_2$ we denote $S_p(\mathcal{H})$ by $S_p^n$ and $S_p$ respectively and we identify $B(\ell_2^n)$ with the set of $n\times n$ matrices which we also denote by $M_n$. The set of all compact operators on $\ell_2$ and $\ell_2^n$ are denoted by $S_\infty$ and $S_\infty^n$ respectively. A linear map $T:L_p(\mathcal{M} )\to L_p(\mathcal{M} )$ is said to be positive if $T$ maps $L_p(\mathcal{M} )_{+}$ to $L_p(\mathcal{M} )_{+}.$ We say that $T$ is \emph{completely positive} if the linear map $I_{S_p^n}\otimes T:L_p(M_n\overline{\otimes}\mathcal{M},Tr\otimes\tau )\to L_p(M_n\overline{\otimes}\mathcal{M},Tr\otimes\tau )$ is positive for all $n\in\mathbb{N}.$ The set of positive and completely positive operators on $L_p(\mathcal M)$ are closed under strong operator limits. A linear map $T:L_p(\mathcal M)\to L_p(\mathcal M)$ is   \emph{completely bounded} if \[\|T\|_{cb, L_p(\mathcal M)\to  L_p(\mathcal M)}\coloneqq \sup\limits_{n\geq 1}\|I_{S_p^n}\otimes T\|_{L_p(M_n\overline{\otimes}\mathcal M,Tr \otimes\tau)\to L_p(M_n\overline{\otimes}\mathcal M,Tr \otimes\tau)} <\infty,\]
and the above quantity is called
the completely bounded (in short c.b.) norm of $T$. Also, $T$ is a \emph{complete contraction} (resp. \emph{complete isometry}) if $I_{S_p^n}\otimes T$ is a contraction (resp. isometry) for all $n\geq 1.$ We say that $T$ is \emph{$n$-contractive} (resp. \emph{$n$-isometry}) for some $n$ if $I_{S_p^n}\otimes T$ is a contraction (resp. {isometry}). We refer to \cite{pisierxu2003nclp} for a comprehensive study of noncommutative $L_p$-spaces and related topics.

\subsection{Noncommutative vector-valued  $L_p$-spaces}\label{sub2.2}

It is well-known that maximal norms on noncommutative $L_p$-spaces require special definitions. This is mainly because the notion $\sup_{n\geq 0}|x_n|$ makes no reasonable sense for a sequence of arbitrary operators $(x_n)_{n\geq 0}.$   This difficulty can be overcome by using the theory of noncommutative vector-valued $L_p$-spaces  which was initiated by Pisier \cite{pisier1998ncvectorLp} and improved by Junge \cite{junge02doob}. For $1\leq p\leq\infty,$ let $L_p(\mathcal M;\ell_\infty)$ be the space of all sequences $x=(x_n)_{n\geq 0}$ admitting the following factorization: there are  $a,b\in L_{2p}(\mathcal M)$ and a bounded sequence $(y_n)_{n\geq 0}\subseteq\mathcal M$ such that $x_n=ay_nb$ for $n\geq 0.$ One defines \[\|(x_n)_{n\geq 0}\|_{L_p(\mathcal M;\ell_\infty)}= \inf\Big\{\|a\|_{2p}\sup_{n\geq 0}\|y_n\|_{\infty}\|b\|_{2p}\Big\}\] where the infimum is taken over all possible factorizations. Let us remark that for any positive sequence $x\in L_p(\mathcal M)$ given by $x=(x_n)_{n\geq 0},$ $x$ belongs to $ L_p(\mathcal M;\ell_\infty)$ if and only if there exists $a\in L_p(\mathcal M)_{+}$ such that $x_n\leq a$ for all $n\geq 0.$ In this case, we have \[\|(x_n)_{n\geq 0}\|_{L_p(\mathcal M;\ell_\infty)}=\inf\{\|a\|_p:x_n\leq a, a\in L_p(\mathcal M)_{+}\}.\] Let $\mathbb{N}_0=\mathbb{N}\cup\{0\}.$ The following folkloric truncated description of the maximal norm is often useful. A proof can be found in \cite{jungexu07erg}.
\begin{prop}\label{FINT}Let $1\leq p\leq\infty.$ A sequence $(x_n)_{n\geq 0}\subseteq L_p(\mathcal M)$ belongs to $L_p(\mathcal M;\ell_\infty)$ if and only if $\sup\limits_{\mathbb N_0\supseteq J\ \text{is finite}}\big\|(x_i)_{i\in J}\big\|_{L_p(\mathcal M;\ell_\infty)}<\infty.$ Moreover, we have that \[\|(x_n)_{n\geq 0}\|_{L_p(\mathcal M;\ell_\infty)}=\sup\limits_{\mathbb N_0\supseteq J\ \text{is finite}}\big\|(x_i)_{i\in J}\big\|_{L_p(\mathcal M;\ell_\infty)}.\]
\end{prop}
Let $1\leq p<\infty.$ We define $L_p(\mathcal M;\ell_1)$ to be the space of all sequences $x=(x_n)_{n\geq 0}\subseteq L_p(\mathcal M)$ which admits a decomposition \[x_n=\sum_{k\geq 0}u_{kn}^*v_{kn}\] for all $n\geq 0,$ where  $(u_{kn})_{k,n\geq 0}$ and $(v_{kn})_{k,n\geq 0}$ are two families in $L_{2p}(\mathcal M)$ such that 
\[\sum_{k,n\geq 0}u_{kn}^*u_{kn}\in L_p(\mathcal M),\quad \sum_{k,n\geq 0}v_{kn}^*v_{kn}\in L_p(\mathcal M).\] In above all the series are required to converge in $L_p$-norm. We equip the space $L_p(\mathcal M;\ell_1)$ with the norm
\[\|x\|_{L_p(\mathcal M;\ell_1)}=\inf\Big\{\Big\|\sum_{k,n\geq 0}u_{kn}^*u_{kn}\Big\|_p^{\frac{1}{2}} \Big\|\sum_{k,n\geq 0}v_{kn}^*v_{kn}\Big\|_p^{\frac{1}{2}}\Big\},\] where infimum runs over all possible decompositions of $x$ described as above. For any positive sequence $x=(x_n)_{n\geq 0}\in L_p(\mathcal M;\ell_1)$ we have a simpler description of the norm as follows
\[\|x\|_{L_p(\mathcal M;\ell_1)}=\Big\|\sum\limits_{n\geq 0}x_n\Big\|_p.\]  
It is known that both $L_p(\mathcal M;\ell_\infty)$ and $L_p(\mathcal M;\ell_1)$ are Banach spaces. Moreover, we have the following duality fact.
\begin{prop}[\cite{junge02doob}] Let $1<p<\infty.$ Let $\frac{1}{p}+\frac{1}{p^\prime}=1.$ Then we have isometrically $L_p(\mathcal M;\ell_1)^*=L_{p^\prime}(\mathcal M;\ell_\infty)$, with the duality relation given by \[\langle x,y\rangle=\sum_{n\geq 0}\tau(x_ny_n)\] for all $x\in L_p(\mathcal M;\ell_1)$ and $y\in L_{p^\prime}(\mathcal M;\ell_\infty).$
\end{prop}

\subsection{Various notions of dilation}
In this subsection, we turn our attention to various   notions of dilations. The study of dilations and $N$-dilations have a long history already for operators on Hilbert spaces (see \cite{sznagyfoics70book,mccarthyshalit13ndilation}), whereas the notion of simultaneous dilation was only recently introduced in \cite{facklergluck19dilation} in the setting of general Banach spaces.
\begin{defn}\label{dide1}Let $1\leq p\leq\infty.$ Let $T:L_p(\mathcal{M},\tau_{\mathcal{M}})\to L_p(\mathcal{M},\tau_{\mathcal{M}})$ be a contraction. We say that $T$ has a \textit{dilation} (resp. \textit{complete dilation}) if there exist  a von Neumann algebra $\mathcal{N}$ with a normal faithful semifinite trace $\tau_{\mathcal{N}},$ contractive  {(resp. completely contractive)} linear maps $Q:L_p(\mathcal{N},\tau_{\mathcal{N}})\to L_p(\mathcal{M},\tau_{\mathcal{M}}),$ $J:L_p(\mathcal{M},\tau_{\mathcal{M}})\to L_p(\mathcal{N},\tau_{\mathcal{N}}),$ and an isometry (resp. complete isometry) $U: L_p(\mathcal{N},\tau_{\mathcal{N}})\to L_p(\mathcal{N},\tau_{\mathcal{N}})$ such that 
	\begin{equation}\label{di}
	T^n=QU^nJ,\quad \forall n\in\mathbb{N}\cup\{0\}. 
	\end{equation}In terms of commutative diagrams, we have
	\begin{displaymath}
	\xymatrix{
		L_p(\mathcal M,\tau_{\mathcal M})\ar[rr]^{T^n}  \ar[d]^{{J}}  &&
		L_p(\mathcal M,\tau_{\mathcal M}) \\
		L_p(\mathcal{N},\tau_{\mathcal{N}}) \ar[rr]^{U^n}&& L_p(\mathcal{N},\tau_{\mathcal{N}})\ar[u]^{{Q}}
	}
	\end{displaymath}
	for all $n\geq 0.$

	We say that $T$ has an \textit{$N$-dilation} if \eqref{di} is true for all $ n\in\{0,1,\dots,N\}$. {We say that $T$ has a \textit{complete $N$-dilation} if \eqref{di} is true for all $ n\in\{0,1,\dots,N\}$ and the isometry $U$ as in \eqref{di} is a complete isometry.}
\end{defn}
\begin{defn}\label{mdil}Let $1\leq p\leq \infty.$ Let $S\subseteq B(L_p(\mathcal{M},\tau_{\mathcal{M}})).$ We say that $S$ has a \textit{simultaneous dilation} (resp. \textit{complete simultaneous dilation}) if  there exist a von Neumann algebra $\mathcal{N}$ with a normal faithful semifinite trace $\tau_{\mathcal{N}},$ contractive (resp. completely contractive) linear maps  $Q:L_p(\mathcal{N},\tau_{\mathcal{N}})\to L_p(\mathcal{M},\tau_{\mathcal{M}}),$ $J:L_p(\mathcal{M},\tau_{\mathcal{M}})\to L_p(\mathcal{N},\tau_{\mathcal{N}}),$ and a set of isometries (resp. complete isometries)  $\mathcal{U}\subseteq L_p(\mathcal{N},\tau_{\mathcal{N}})$ such that for all $n\in\mathbb{N}\cup\{0\}$ and $T_i\in S,$ $1\leq i\leq n,$ there exist  $U_{T_1},U_{T_2},\dots,U_{T_n}\in\mathcal{U}$  such that
	\begin{equation}\label{di1}
	T_1T_2\dots T_n=QU_{T_1}U_{T_2}\dots U_{T_n}J.
	\end{equation}In terms of commutative diagrams, we have
	\begin{displaymath}
	\xymatrix{
		L_p(\mathcal M,\tau_{\mathcal M})\ar[rr]^{T_1\dots T_n}  \ar[d]^{{J}}  &&
		L_p(\mathcal M,\tau_{\mathcal M}) \\
		L_p(\mathcal{N},\tau_{\mathcal{N}}) \ar[rr]^{U_{T_1}\dots U_{T_n}}&& L_p(\mathcal{N},\tau_{\mathcal{N}})\ar[u]^{{Q}}.
	}\end{displaymath} The empty product (i.e. $n=0$) corresponds to the identity operator.

	We say that $S$ has a \emph{simultaneous $N$-dilation} if \eqref{di1} is true for all $n\in\{0,1,\dots,N\}$. We say that $S$ has a \emph{complete simultaneous $N$-dilation} if \eqref{di1} is true for all $n\in\{0,1,\dots,N\}$ and the family $\mathcal U$ consists of complete isometries.
\end{defn}
\begin{rem}Let $1\leq p\leq\infty$. If $S\subseteq B(L_p(\mathcal{M},\tau_{\mathcal{M}}))$ has a simultaneous (resp. complete simultaneous) $N$-dilation for any $N\in\mathbb{N} $, then  for any $n\geq 1$ and $T_1,\dots,T_n\in S$, the operator $T_1\dots T_n$ has a simultaneous (resp. complete simultaneous) $N$-dilation for any $N\in\mathbb{N}.$
\end{rem}

\subsection{Characterization theorems for isometries and complete isometries}
We recall the definition of a Jordan homomorphism.
A complex linear map $J:\mathcal{M}\to\mathcal{N}$ is called a Jordan $*$-homomorphism if $J(x^2)=J(x)^2 $ and $J(x^*)=J(x)^* $ for all $x\in\mathcal{M}.$ It is well-known that in this case $J(xyx)=J(x)J(y)J(x)$ for all $x,y\in \mathcal M$.
\begin{lem}[{\cite{stormer65jordan}}]\label{decom}Let $J:\mathcal{M}\to\mathcal{N}$ be a normal Jordan $*$-homomorphism. Let $\widetilde{\mathcal N}$ denote the von Neumann subalgebra generated by $J(\mathcal M)$ in $\mathcal{N}$. Then there exist two central projections $e,f\in \mathcal Z  (\widetilde{\mathcal N}) $ with $e+f=1_{\widetilde{\mathcal N}} $ such that $x\mapsto J(x)e$ is a $*$-homomorphism and $x\mapsto J(x)f$ is a $*$-anti-homomorphism.
\end{lem}

We should warn the reader that in the above theorem, $J(\mathcal M )$ is in general \emph{not} necessarily a von Neumann subalgebra of $\mathcal{N}$. However it is stable under the usual Jordan product and is still a $\operatorname{w}^*$-closed subspace of $\mathcal{N}$. We refer to \cite[Section 4.5]{hanchestormer84jordan} and the references therein for more details.

The following structural description of isometries and complete isometries will be frequently used. We refer to \cite{lamperti58isom} for the classical case.
\begin{thm}[\cite{yeadon81isom,jungeruansherman05isom}]\label{ISOthm}
	Let $1\leq p\neq 2<\infty.$ Let $T:L_p(\mathcal M,\tau_{\mathcal M})\to L_p(\mathcal N,\tau_{\mathcal N})$ be a bounded operator. Then $T$ is an isometry if and only if there exist uniquely a normal Jordan $*$-monomorphism  $J:\mathcal M\to \mathcal N,$ a partial isometry $w\in \mathcal N,$ and a positive self-adjoint operator $b$ affiliated with $\mathcal N,$ such that the following hold:
	\begin{enumerate}
		\item[(i)]$w^*w=s(b)=J(1);$
		\item[(ii)] Every spectral projection of $b$ commutes with $J(x)$ for all $x\in\mathcal M;$
		\item[(iii)] $T(x)=wbJ(x)$ for all $x\in\mathcal S(\mathcal M);$
		\item[(iv)]$\tau_{\mathcal N}(b^pJ(x))=\tau_{\mathcal M}(x)$ for all $x\in\mathcal M_+.$
	\end{enumerate}
	Moreover, $T$ is a complete isometry if and only if the Jordan $*$-monomorphism $J$ as above is multiplicative.
\end{thm}

The following property is kindly communicated to us by Arhancet.
\begin{thm}\cite{arhancet20characterization} \label{thm:cedric}
	Let $1\leq p<\infty.$	Let $T:L_p(\mathcal M,\tau_{\mathcal M})\to L_p(\mathcal N,\tau_{\mathcal N})$ be a positive isometry of the form $T=wbJ$ where $w,b,J$ are objects as in Theorem \ref{ISOthm}. Then $T$ is completely positive if and only if it is $2$-positive if and only if the Jordan $*$-monomorphism $J$ is multiplicative. 
\end{thm}

\section{Lamperti operators on noncommutative $L_p$-spaces}\label{SLAM} In this section, we establish some elementary properties and prove two structural theorems for Lamperti and completely Lamperti operators respectively. Our study is motivated by the argument for the particular case of isometries, see for instance \cite{yeadon81isom}.

Let us start with some useful  properties of Lamperti operators. In the commutative setting, similar results were established in \cite{kan78erglamperti} (see \cite{kan79thesis} for detailed proofs). 
Before the discussion we recall the following elementary fact.
\begin{lem}\label{1.10}Let $1\leq p<\infty.$ Let $x\in L_p(\mathcal M)_{+}.$ Then there exists a sequence $(x_n)_{n\geq 1}\subseteq \mathcal{S}(\mathcal  M)_{+}$ such that $x_n\leq x$, $\lim\limits_{n\to\infty}\|x_n-x\|_p=0$ and $s(x_n)\uparrow s(x).$ Moreover, if $y\in L_p(\mathcal M)_{+}$ satisfies $xy=0,$ then we can choose a sequence $(y_n)_{n\geq 1}\subseteq \mathcal{S}(\mathcal M)$ as described for $x$ such that $x_ny_n=0$ for all $n\geq 1.$ 
\end{lem}
\begin{proof} The first assertion follows from the corresponding commutative case by considering the abelian von Neumann subalgebra generated by the spectral resolution of $x$. For the second assertion, it suffices to notice that if $xy=0$ for $x,y\in L_p (\mathcal M )_+$, then $s(x)s(y)=0$ by a standard argument of functional calculus, and vice versa.
\end{proof}	

The lemma immediately yields the following property.
\begin{prop}\label{NOTS}Let $1\leq p<\infty.$ A positive bounded linear map $T:L_p(\mathcal{M})\to L_p(\mathcal{M})$ is 
	Lamperti if and only if for any $x,y\in L_p(\mathcal  M)_+$ with $xy=0,$ we have $TxTy=0.$ In this case we have
	\[|Tx|=T(|x|),\quad x=x^*,\ x\in L_p(\mathcal{M}).\]
		
	In particular, if both $T_1$ and $T_2$ are positive Lamperti operators on $L_p	(\mathcal{M})$, then $T_1 T_2$ is also positive Lamperti.
\end{prop}
\begin{proof}One direction is clear. Now  let us begin with a positive Lamperti operator $T:L_p(\mathcal{M})\to L_p(\mathcal{M}).$ Let $x,y\in L_p(\mathcal  M)_+$ with $xy=0$. Using Lemma \ref{1.10}, we obtain sequences $(x_n)_{n\geq 1},(y_n)_{n\geq 1}$ in $\mathcal{S}(\mathcal  M)_+$ such that $\|x_n-x\|_p\to 0$ and $\|y_n-y\|_p\to 0$ and $x_ny_n=0$ for all $n\geq 1.$ Since $T$ is Lamperti, we can easily verify that $Tx_nTy_n=Ty_nTx_n=0$ for all $n\in\mathbb{N}$. Therefore, by \cite[Theorem 1]{yeadon81isom} for $p\neq 2$ and by the parallelogram law for $p=2$ we have \[\|Tx_n+Ty_n\|_p^p+\|Tx_n-Ty_n\|_p^p=2(\|Tx_n\|_p^p+\|Ty_n\|_p^p).\] 
	Taking limit, we have
	\[\|Tx +Ty \|_p^p+\|Tx -Ty \|_p^p=2(\|Tx \|_p^p+\|Ty \|_p^p).\]  
	For $p\neq 2$, again applying \cite[Theorem 1]{yeadon81isom}  we obtain that $TxTy=TyTx=0.$
	For $p=2$,	the above equality in turn implies $\tau (TxTy)=0.$ Thus, $(Tx)^{\frac{1}{2}} Ty (Tx)^{\frac{1}{2}} =0 $. In other words we have
	\[ ( (Ty)^{\frac{1}{2}} (Tx)^{\frac{1}{2}} )^* ( (Ty)^{\frac{1}{2}} (Tx)^{\frac{1}{2}}) =0,\]
	whence
	$ (Ty)^{\frac{1}{2}} (Tx)^{\frac{1}{2}} =0.$
	Therefore, we conclude
	$ TxTy=0$.
	
	Let $x\in L_p (\mathcal M)$ be a self-adjoint element. Decompose $x$ as $x=x^+-x^-$. Since $x^+x^{-}=0,$ we see that   $T(x^+){T(x^-)}=0.$ This implies that $|T(x)|=T(x^+)+T(x^-)=T(|x|).$
	This completes the proof of the proposition.
\end{proof}

Now we state the main result of this section.	

\begin{thm}\label{charac} Let $1\leq p<\infty.$ Let $T:L_p(\mathcal{M},\tau )\to L_p(\mathcal{M},\tau )$ be a Lamperti operator with norm $C$. Then  there exist, uniquely, a partial isometry $w\in\mathcal{M},$ a  positive self-adjoint operator $b$ affiliated with $\mathcal{M}$ and a normal Jordan $*$-homomorphism $J:\mathcal{M}\to\mathcal{M},$ such that
	\begin{enumerate}
		\item $w^*w=J(1)=s(b)$; moreover we have $w=J(1)=s(b)$ if additionally $T$ is  positive;
		\item Every spectral projection of $b$ commutes with $J(x)$ for all $x\in\mathcal{M};$
		\item $T(x)=wbJ(x),$ $x\in\mathcal{S}(\mathcal{M});$
		\item We have  $\tau (b^pJ(x))\leq C\tau (x) $ for all $x\in \mathcal{M}_+$; if additionally $T$ is isometric, then the equality holds with $C=1$.
	\end{enumerate}
\end{thm}
\begin{rem}\label{kotokikorarchilobaki}
	Note that any operator $T$ defined on $\mathcal {S}(\mathcal M)$ satisfying (i)-(iv) in Theorem \ref{charac} can be extended to a Lamperti operator with $\|T\|_{L_p(\mathcal M)\to L_p(\mathcal M)}\leq (2C)^{\frac{1}{p}}$ (or $\leq C^{\frac{1}{p}}$ if $J$ is additionally a normal $*$-homomorphism) Indeed, recall that by Lemma \ref{decom}, $J:\mathcal{M}\to\mathcal{M}$ can be written as a direct sum $J=J_1 +J_2,$ where $J_1$ is a $*$-homomorphism, $J_2$ is a $*$-anti-homomorphism and the images of $J_1$ and $J_2$ commute.  Without loss of generality, assume $C=1$ and note that for $x\in \mathcal{S}(\mathcal M) ,$ we have  \begin{equation}\label{eq:txp}
	|T(x)|^p=b^p|J(x)|^p = b^p(J_1(|x|^p)+J_2(|x^*|^p)). 
	\end{equation} 
	Note also that 
	\[\tau(b^pJ_1(|x|^p)) = \tau(b^pJ (|x|^p) e ) \leq \tau(b^pJ (|x|^p)  ) \] and similar inequality holds for $J_2$. Therefore, by (\emph{iv}) we have \[\tau(|T(x)|^p)=\tau(b^pJ_1(|x|^p)+\tau(b^pJ_2(|x^*|^p))\leq 2\|x\|_p^p.\] 
	Thus $T$ can be extended to a bounded operator on $L_p(\mathcal M)$.
	On the other hand, take two $\tau$-finite projections $e,f$  with $ef=0.$ Then we have \[(Te)^*Tf=J(e)bw^*wbJ(f)=J(e)bJ(1)bJ(f)=b^2J(ef)=0.\]In above, we have used the fact that for a Jordan $*$-homomorphism, $J(xy)=J(x)J(y)$ whenever $x$ and $y$ commutes. So $T$ is also Lamperti.
\end{rem}	
Now we give the proof of Theorem \ref{charac}. Our strategy is adapted from \cite{yeadon81isom}. However, a few key steps such as the verification of normality of $J$ turn out to be different in our new setting, so we would like to include a complete proof for this result. 
\begin{proof}
	Without loss of generality, assume that $T$ is a Lamperti contraction. We first construct the related objects for self-adjoint elements in $\mathcal{S}(\mathcal{M})$. 
	To begin with, for any projection $e\in\mathcal{S}(\mathcal{M}),$ 
	we choose a partial isometry $w_e \in \mathcal M$, a positive operator $b_e \in L_0(\mathcal M)$ and a projection $J(e)\in \mathcal M$ by using the polar decomposition: 
	\[Te=w_eb_e, \quad b_e =|Te|,\quad J(e)=w_e ^* w_e=s(b_e).\]
	Note that for two finite projections $e,f\in \mathcal{M} $ with $ef=0$, we have $(Te)^*Tf=Te(Tf)^*=0 $ by the Lamperti property of $T$, whence $b_e w_e^*w_f b_f=w_e b_e b_f  w_f^*=0$. Multiplying $w_e^*$ and $w_f$, we get $b_e b_f =0$. Then it is routine to check $(Te+Tf)^* (Te +Tf) = (|Te|+|Tf|)^2$. In other words  we get  \begin{equation}\label{eq:add b}
	b_{e+f}=b_e+b_f.
	\end{equation} 
	Recall that $b_e b_f = b_f b_e =0$. By considering the commutative von Neumann subalgebra generated by the spectral projections, we see that the supports of $b_e$ and $b_f$ are also disjoint and additive, that is,
	\begin{equation}\label{eq:add j}
	J(e+f)=J(e)+J(f), \quad J(e)J(f)=J(f)J(e)=0.
	\end{equation} 
	Moreover if we denote by $x^{-1}\in L_0 (\mathcal M)$ the element given by the functional calculus associated with $t\mapsto t^{-1} \chi_{\{t>0\}}$, then $b_{e+f}^{-1} = b_{e}^{-1} +b_{f}^{-1}$. So we may write 
	\[T(e+f) b_{e+f}^{-1}= w_e b_e b_{e+f}^{-1} + w_f b_f b_{e+f}^{-1} = w_e s(b_e) + w_f s(b_f)
	=w_e +w_f,\] which means that
	\begin{equation}\label{eq:add w}
	w_{e+f} = w_e +w_f.
	\end{equation}
	More generally, if a self-adjoint element $x\in \mathcal S (\mathcal M )$ is of the form    
	\begin{equation}\label{eq:simple}
	x=\sum_{i=1}^n\lambda_ie_i,\quad \lambda_i\in\mathbb{R},
	\end{equation} 
	where   $e_i$'s are some $\tau$-finite projections in $\mathcal{M}$ with $e_i e_j =0$ for $i\neq j$, then we define \[J(x)=\sum_{i=1}^n\lambda_i J(e_i).\]
	From \eqref{eq:add j} we see that for any two commuting self-adjoint operators $x, y $ of the above form, we have
	\begin{itemize}\label{propj}
		\item[(a)] $J(x^2)=J(x)^2;$
		\item[(b)] $\|J(x)\|_{\infty}\leq\|x\|_{\infty};$
		\item[(c)] $J(\lambda x+y)=\lambda J(x)+J(y),$ $\lambda\in\mathbb R.$
	\end{itemize}
	Moreover, for a self-adjoint element $x=x^*\in \mathcal{S}(\mathcal{M}),$ we take a sequence of step functions $f_n $ with $f_n(0)=0$ converging uniformly to the identity function $1(\lambda)=\lambda$ on the spectrum of $x,$ then the  element $f_n (x)$ is of the form \eqref{eq:simple} and we define \[J(x)=\lim_n J(f_n(x))\] in $\|\cdot\|_\infty$ norm in $\mathcal{M}.$ This limit exists and is independent of the choice of the sequence because of the above property (b) of the map $J.$ Note that now the assertions (a), (b) and (c) also hold for all self-adjoint elements in $\mathcal{S}(\mathcal{M})$. 
	
	We will check that $J$ is real linear and hence we may extend $J$ as a complex linear map to the whole space $\mathcal{S}(\mathcal M)$. Let $f\leq e $ be two projections in $\mathcal{S}(\mathcal{M}).$ Note that $T(f)J(f)=T(f)$ and $T(e-f)J(f)=0.$ Therefore  $T(f)=T(e)J(f).$ Thus by the linearity of $T$ and the assertion (c), we have $T(x)=T(e)J(x)$  for all self-adjoint elements $x\in\mathcal{S}(\mathcal M)$ of the form \eqref{eq:simple} with $s(x)\leq e.$ 
	Using the approximation by step functions $f_n $ as before, we obtain \[\|T(e)(J(x)-J(f_n(x)))\|_p\leq \|T(e)\|_p\|J(x)-J(f_n(x))\|_\infty \] 
	and hence  \begin{equation}\label{eq:tj}
	w_e b_e J(x)=T(e)J(x)=\lim\limits_{n\to\infty}T(e)J(f_n(x))=\lim\limits_{n\to\infty}T(f_n(x))=T(x),
	\end{equation} 
	where the limit is taken in $\|\cdot\|_p$ norm and we have used the fact that $x=\lim_{n\to\infty}f_n(x)$ in $\|\cdot\|_p$ norm for $x\in \mathcal{S}(\mathcal{M})$. Thus for any two self-adjoint operators $x,y\in\mathcal{S}(\mathcal{M}) $ with $e=s(x)\vee s(y),$ we have \[T(e)(J(x+y)-J(x)-J(y))=T(x+y)-T(x)-T(y)=0.\] 
	Note that $J(x+y)-J(x)-J(y)$ has the range projection contained in the support projection $J(e)$ of $T(e) $, which yields  \[J(x+y)=J(x)+J(y) ,\]
	as desired.		
	By the real linearity, we may extend $J$ as a continuous complex linear map (in $\|\cdot\|_\infty$ norm) on $\mathcal{S}(\mathcal M)$ as \[J(x+iy)=J(x)+iJ(y), \quad x,y\in \mathcal{S}(\mathcal M) \text{ self-adjoint} .\]
	Note that in this setting we also have  
	\begin{equation}\label{eq:jordan}
	J(x^*)=J(x)^*,\quad   J(x^2)=J(x)^2,\quad x\in\mathcal{S}(\mathcal{M}). 
	\end{equation}
	
	Now we check the commutativity of $b_e$ and $J(x)$ for $x\in\mathcal{S}(\mathcal{M})$ with $s(x)\leq e$. For $\tau$-finite projections $e,f\in\mathcal{M} $ with $f\leq e,$ by definition we see that $b_{e-f}J(f)=0$ and $b_fJ(f)=b_f.$ Together with \eqref{eq:add b} we get $b_eJ(f)=b_f=J(f)b_e.$ As a consequence $b_e$ commutes
	with $J(x)$ for all $x$ of the form \eqref{eq:simple}. By an approximation argument as before, we may find a sequence of elements $(x_n)$ of the form  \eqref{eq:simple} so that
	\begin{equation}\label{eq:bj}
	b_e J(x) = \lim_{n\to\infty} b_e J(x_n) = \lim_{n\to\infty} J(x _n )b_e = J(x)b_e ,  
	\end{equation}
	where the limit has been taken in $\|\cdot\|_p$ norm. Therefore, we obtain the desired commutativity.
	
	Moreover, we see that \begin{equation}\label{nor1} \tau (b_e^pJ(x)) \leq \tau (x ),\end{equation} whenever $s(x)\leq e$, $x\in\mathcal M_{+}$ and the equality holds if $T$ is an isometry.  Indeed, by  \eqref{eq:tj} and the commutativity between $b_e$ and $J(x)$, we see that 
	$\tau( |T(x)|^p) = \tau (b_e^p J(x)^p) = \tau (b_e^p J(x ^p) )$. However
	$\tau( |T(x)|^p)\leq \tau(x^p)$  since $T$ is a contraction. Thus we obtain $\tau (b_e^p J(x ^p) )  \leq \tau (x ^p )$. Note that $x$ is arbitrarily chosen, so the inequality \eqref{nor1} is proved.
	
	The rest of the proof splits into the following two steps:
	
	(1) Case where $\tau$ is finite: In this case we have $\mathcal{S}(\mathcal{M}) = \mathcal M $ and we take $w = w_1$ and $b = b_1 $. Together with the construction and the properties \eqref{eq:tj}-\eqref{nor1}, the proof is complete except the normality of $J,$ which we prove now. Take a bounded increasing net of positive operators $(x_\alpha)$ strongly converging to $x $, and let $a$ be the supreme of  $(J(x_\alpha))$. By \eqref{nor1}, we have  $\tau(b ^pJ(x-x_\alpha))\leq \tau(x-x_\alpha)\to 0.$  Therefore,  we obtain  \begin{equation}\label{eq:taubjx}
	\lim_\alpha\tau(b ^pJ(x_\alpha))=\tau(b ^pJ(x)).  
	\end{equation}
	Also, note that $b^p \in L_1 (\mathcal M)_+ $ since by \eqref{nor1} we have 
	\begin{equation}\label{eq:bplone}
	\tau (b^p) =\tau(b^ps(b))  =\tau(b^pJ(1)) \leq \tau(1)<\infty. 
	\end{equation} 
	Thus $x\mapsto \tau (b^p x)$ is a normal functional. Therefore by the definition of $a$, we also have \[\lim_\alpha \tau(b ^pJ(x_\alpha))= \tau(b ^pa).\] Together with \eqref{eq:taubjx} this implies that $\tau(b ^pa)=\tau(b ^pJ(x)).$ Note that $J$ is positive according to \eqref{eq:jordan}, so $J(x_\alpha)\leq J(x)$ and consequently $a\leq J(x)$. In other words, we obtain \[b ^{\frac{p}{2}}(J(x)-a)b ^{\frac{p}{2}}\geq 0  \quad \text{ but }\quad  \tau(b ^{\frac{p}{2}}(J(x)-a)b ^{\frac{p}{2}})=0,\]
	which yields $b ^{\frac{p}{2}}(J(x)-a)b ^{\frac{p}{2}} =0$ by the faithfulness of $\tau$. Recall that $J(1)=s(b)$, so we have $J(1)(J(x)-a)J(1)=0$, that is, $ J(x)  =J(1)aJ(1) $. However, we observe that \[J(1)aJ(1)=\lim_{\alpha}J(1)J(x_\alpha)J(1)=\lim_{\alpha}J(x_\alpha)=a.\] Thus, we obtain $a=J(x)$ which implies that $J$ is normal.
	
	(2) Case where  $\tau $ is not finite:   
	Denote by $\mathcal F$ the net of all $\tau$-finite projections in $\mathcal{M}$ equipped with the usual upward partial order. Then this net converges to $1$ in the strong operator topology. For any $x\in\mathcal{M}$, if $e,f\in\mathcal F$ with $e \leq f$, then 
	\[J(e  x e)=J(e )J(f xf )J(e)\] since we have already proved in Case (1) that the restriction of $J$ on the reduced von Neumann subalgebra $f\mathcal M f$ is a Jordan $*$-homomorphism. 
	Note that by the construction of $J$,   $(J(e))_{e\in \mathcal F}$ is also an increasing net of projections, so it converges to $J(1)\coloneqq\sup_e J(e )$ in the strong operator topology.
	Thus the above relation shows that the net $(J(e  xe ))_{e\in\mathcal F}$ converges in the strong operator topology. We denote this limit by
	\begin{equation}\label{defnref}
	J(x) =\lim_{e\in \mathcal F} J(e  xe ).\end{equation}
	Note that this also yields
	\begin{equation}\label{eq:exe}
	J(e  xe )=J(e )J(x)J(e ),\quad e\in\mathcal F ,x\in \mathcal M. 
	\end{equation} 				
	We obtain a linear map $J:\mathcal M \to \mathcal M$. We show that it is a normal Jordan $*$-homomorphism. It is normal since for any bounded monotone net $(x_i)_{i\in I}\subseteq \mathcal M_+$  and for any  $e\in\mathcal F$,
	\[J(e ) (\sup_i J(x_i)) J(e ) = \sup_i J(e  x_i e ) = J(e (\sup_i x_i) e) =J(e) J(\sup_i x_i) J(e),\]
	where we have used \eqref{eq:exe} and the fact that $J$ is normal on the finite von Neumann subalgebra $e \mathcal M e$ proved in Case (1).
	Hence $\sup\limits_i J(x_i) =J(\sup\limits_i x_i)$. Similarly  $J(x)^* = J(x^*)$ for all $x\in\mathcal M$. On the other hand, we note that for a self-adjoint element $x\in \mathcal M$, the net $(xe x)_{e\in\mathcal F}$ is increasing and bounded. Hence by the normality of $J$ and the relations \eqref{eq:exe} and \eqref{eq:jordan},  we obtain that for any $f\in \mathcal F$,  
	\begin{align*}
	J(f)J(x^2)J(f)& =\sup_{e\in\mathcal F} J(f ) J(xe x) J(f)
	=\lim_{e\in\mathcal F}   J(f  e)  J(xe x)  J(e f ) \\
	&=\lim_{e\in\mathcal F} J(f )J(e xe x e)J(f )
	=\lim_{e\in\mathcal F} J(f )J(e xe )^2 J(f )
	\\&=J(f )J(x)^2 J(f ),
	\end{align*}
	where the limit is taken with respect to the strong operator topology. Hence $J(x^2) = J(x)^2 $.
	
	Also, note that by \eqref{eq:add w} and the definition of $w_e$ and $J$, we have $w_{e}=w_{f}J(e)$ for $e\leq f$ in $\mathcal F$, so we may define similarly
	\[w=\lim_{e\in\mathcal F} w(e)\]
	where the limit is taken with respect to the strong operator topology. Thus we also have $w_{e} =wJ(e)$ and $w^* w = J(1)$. 
	
	For the definition of $b$, we consider the spectral resolution  $b_{e}=\int_0^\infty\lambda dP_{e}(\lambda) $. Clearly, $J(e)=1-P_{e}(0).$ As mentioned earlier, $b_f=b_eJ(f)$ for two $\tau$-finite projections $f\leq e.$ Therefore,  for $\lambda\geq 0$  and $\tau$-finite projections $f \leq e,$ we have $1-P_{f}(\lambda)=(1-P_{e}(\lambda))J(f).$ 
	As before, we can define $P(\lambda)$ to be the limit of $P_{e}(\lambda)$ in the strong operator topology. We set 
	\[b= \int_0^\infty\lambda dP(\lambda),\]
	which is obviously a positive self-adjoint operator affiliated with $\mathcal M$.
	Therefore, we can deduce that $1-P_{e}(\lambda)=(1-P (\lambda))J(e)$ and $b_e = b J(e) $ as well.
	
	As a result we have constructed  a partial isometry $w$, a positive self-adjoint operator $b$ and a normal Jordan $*$-homomorphism $J$. Let us check that they satisfy the properties {(i)-(iv)} stated in the theorem. The assertion {(i)} follows simply from an approximation argument and the fact 
	\[s(b)=1-P(0)=1 - \lim_{e\in\mathcal F} P_{e} (0) = \lim_{e\in\mathcal F} s(b_{e}) = \lim_{e\in\mathcal F} J(e) = J(1).\]
	The assertion {(ii)} follows again by an approximation argument and from the fact that $P(\lambda)$ commutes with $J(e)$ for all $\lambda$ and $e\in\mathcal F$. To see the assertion {(iii)}, it suffices to recall $w_e=wJ(e)$,  $b_e = b J(e) $ and the relation \eqref{eq:tj} for $e=s(x)$.   
	For the assertion {(iv)}, note that the weight $x\mapsto\tau(b^pJ(x))$ is well-defined on $\mathcal M _+$ and is normal since $\tau$ extends to $\widehat{\mathcal M}_+$ with the property $\tau(\sup_i x_i ) = \sup_i \tau(x_i)$ for all increasing net $(x_i)$ in $\widehat{\mathcal M}_+$ (see e.g. \cite[Chap.IX, Corollary 4.9]{takesaki2003oa2}). Now let us take an increasing sequence of spectral projections of $x,$  $(e_n)_{n\in\mathbb N}\subseteq \mathcal F$ so that $e_n$ converges to $s(x)$ strongly. Then  we have for all $n$, 
	\begin{equation}\label{eq:taubjgeneral}
	\tau(b^pJ(x)J(e_n))=\tau(b^pJ(e_n)J(x)J(e_n))=\tau(b_{e_n}^pJ(e_n xe_n))\leq \tau(e_n xe_n).
	\end{equation}  
	Letting $n$ tend to infinity,  
	we have  $\tau(b^pJ(x))\leq \tau(x) $, where the equality holds if additionally $T$ is isometric. So {(iv)} is proved.
	
	If in addition  $T$ is positive, then for any projection $e\in \mathcal S  (\mathcal M)$, by definition we have $b_e = |Te| =Te $ and $w_e$ is the orthogonal projection onto $\overline{\operatorname{ran}\,(Te)}$. Hence $w = \lim_e w_e$ is also an orthogonal projection and therefore $w=w^*w=J(1)=s(b)$. 
	
	The uniqueness of $w,b$ and $J$ is proved in the same way as in  \cite{yeadon81isom}. We omit the details. This completes the proof of the theorem.
\end{proof}
\begin{rem} It was kindly pointed out to us by the anonymous referee that instead of \eqref{defnref}, one can use Theorem 8.1.1 of \cite{extension of J Ham} to extend $J$ to the whole von Neumann algebra $\mathcal{M}$, whenever $\mathcal M$ does not contain any type $I_2$ direct summand. Also the normality of $J$ follows from the same theorem.
\end{rem}
\begin{rem}
	We may also observe that a similar characterization of Lamperti operators   $T:L_p(\mathcal M,\tau_{\mathcal M})\to L_p(\mathcal N,\tau_{\mathcal N})$ between different $L_p$-spaces ($1\leq p<\infty$) can be obtained easily from the above proof. 
\end{rem}

The following theorem is an adaption of the argument presented in \cite{jungeruansherman05isom} in the case of complete isometries. A Lamperti operator  $T:L_p(\mathcal{M} )\to L_p(\mathcal{M} )$ is said to be \emph{$2$-Lamperti} or \emph{$2$-support separating} if the linear map $I_{S_p^2}\otimes T:L_p(M_2\overline{\otimes}\mathcal{M} )\to L_p(M_2\overline{\otimes}\mathcal{M} )$ also extends to a Lamperti operator; it is said to be
\emph{completely Lamperti} (or  \emph{completely support separating})    if for all $n\in\mathbb{N}$, the linear map $I_{S_p^n}\otimes T:L_p(M_n\overline{\otimes}\mathcal{M},Tr_n\otimes\tau)\to L_p(M_n\overline{\otimes}\mathcal{M},Tr_n\otimes\tau)$ extends to a Lamperti operator. 	

\begin{thm}\label{scharac}Let $1\leq p<\infty.$
	Let $T:L_p(\mathcal{M},\tau)\to L_p(\mathcal{M},\tau)$ be a Lamperti operator. Then the following assertions are equivalent:
	\begin{enumerate}
		\item $T$ is completely Lamperti;
		\item $T$ is $2$-Lamperti;
		\item The map $J$ in Theorem \ref{charac} is actually a $*$-homomorphism.
	\end{enumerate}
	In this case we have $ \|T\|_{cb,\,  L_p(\mathcal M)\to  L_p(\mathcal M)}= \|T\|_{ L_p(\mathcal{M})\to  L_p(\mathcal{M})} $.
\end{thm}
\begin{proof} Note that {(i)$\Rightarrow$(ii)} is trivial.
	
	We now prove {(ii)$\Rightarrow$(iii)}. Let us denote $\mathbf{T}_2 = I_{S_p^2}\otimes T:L_p(M_2\overline{\otimes}\mathcal{M})\to L_p(M_2\overline{\otimes}\mathcal{M}).$ Since $\mathbf{T}_2$ separates supports, by Theorem \ref{charac} there exists a partial isometry $\widetilde{w}\in M_2\overline{\otimes}\mathcal{M},$ a positive self-adjoint operator $\widetilde{b}$ affiliated with $M_2\overline{\otimes}\mathcal{M}$ and a normal Jordan $*$-homomorphism $\widetilde{J}:M_2\overline{\otimes}\mathcal{M}\to M_2\overline{\otimes}\mathcal{M}$ such that $\widetilde{w}^*\widetilde{w}=\widetilde{J}(1_{M_2}\otimes 1)=s(\widetilde{b}),$  every spectral projection of $\widetilde{b}$ commutes with $\widetilde{J}(\widetilde{x})$ for all $\widetilde{x}\in M_2\overline{\otimes}\mathcal{M},$ and $\mathbf{T}_2(\widetilde{x})=\widetilde{w}\widetilde{b}\widetilde{J}(\widetilde{x}),$ $\widetilde{x}\in \mathcal{S}(M_2\overline{\otimes}\mathcal{M}).$ Also, $T$ separates supports. Thus, again by Theorem \ref{charac}, $Tx=wbJ(x),$ $x\in\mathcal{S}(\mathcal M)$ with $w,b$ and $J$ as in Theorem \ref{charac}. Let us consider two $\tau$-finite projections $e_1,e_2$ in $\mathcal{M}.$ Clearly, $\widetilde{e}=\left(
	\begin{array}{ccccc}
	e_1 & 0\\
	0 &  e_2\\
	\end{array}
	\right)$ is a $Tr\otimes\tau$-finite projection in $M_2\overline{\otimes} \mathcal{M}.$ Let $\mathbf{T}_2(\widetilde{e})=\widetilde{w_e}\widetilde{b_e}$ with $|\mathbf{T}_2(\widetilde{e})|=\widetilde{b_e} $ be the polar decomposition of $\mathbf{T}_2(\widetilde{e}) $ and $T(e_i)=w_{e_i}b_{e_i}$ with $|T(e_i)|=b_{e_i}$ be that of $T(e_i)$ for $  i\in\{1,2\}.$ Note that 
	\begin{equation*} 
	\mathbf{T}_2(\widetilde{e})=\left(
	\begin{array}{ccccc}
	T(e_1) & 0\\
	0 &  T(e_2)\\
	\end{array}
	\right)=\left(
	\begin{array}{ccccc}
	w_{e_1} & 0\\
	0 &  w_{e_2}\\
	\end{array}
	\right)\left(
	\begin{array}{ccccc}
	b_{e_1} & 0\\
	0 &  b_{e_2}\\
	\end{array}
	\right).
	\end{equation*}
	By the uniqueness of the polar decomposition, we have \[\widetilde{w_e}=\left(
	\begin{array}{ccccc}
	w_{e_1} & 0\\
	0 &  w_{e_2}\\
	\end{array}
	\right) \quad
	\text{and}\quad 
	\widetilde{b_e}=\left(
	\begin{array}{ccccc}
	b_{e_1} & 0\\
	0 &  b_{e_2}\\
	\end{array}
	\right).\]
	By the definition of $\widetilde{J}$ as in the proof of Theorem \ref{charac} and by uniqueness, we must have \[\widetilde{J}(\left(
	\begin{array}{ccccc}
	e_1 & 0\\
	0 &  e_2\\
	\end{array}
	\right))=\left(
	\begin{array}{ccccc}
	J(e_1) & 0\\
	0 &  J(e_2)\\
	\end{array}
	\right).\] From this we can easily conclude that $\widetilde{J}(\left(
	\begin{array}{ccccc}
	x & 0\\
	0 &  y\\
	\end{array}
	\right))=\left(
	\begin{array}{ccccc}
	J(x) & 0\\
	0 &  J(y)\\
	\end{array}
	\right)$ for all $x,y\in\mathcal{S}(\mathcal{M}).$ Note that $\mathbf{T}_2$ is an $M_2$-bimodule morphism. Therefore,  we have
	\[\mathbf{T}_2\Big(\left(
	\begin{array}{ccccc}
	0 & x\\
	y &  0\\
	\end{array}
	\right)\Big)=\mathbf T_2\Big(\left(
	\begin{array}{ccccc}
	0 & 1\\
	1 &  0\\
	\end{array}
	\right)\left(
	\begin{array}{ccccc}
	y & 0\\
	0 &  x\\
	\end{array}
	\right)\Big)=\left(
	\begin{array}{ccccc}
	0 & 1\\
	1 &  0\\
	\end{array}
	\right)\left(
	\begin{array}{ccccc}
	T(y) & 0\\
	0 &  T(x)\\
	\end{array}
	\right).\] 
	In other words,
	\[\left(
	\begin{array}{ccccc}
	w  & 0\\
	0 &  w \\
	\end{array}
	\right)\left(
	\begin{array}{ccccc}
	b  & 0\\
	0 &  b \\
	\end{array}
	\right)\widetilde{J}\Big(\left(
	\begin{array}{ccccc}
	0 & x\\
	y &  0\\
	\end{array}
	\right)\Big)=\left(
	\begin{array}{ccccc}
	w & 0\\
	0 &  w \\
	\end{array}
	\right)\left(
	\begin{array}{ccccc}
	b  & 0\\
	0 &  b \\
	\end{array}
	\right)\left(
	\begin{array}{ccccc}
	0 & J(x)\\
	J(y) &  0\\
	\end{array}
	\right).\] Together with the relation $w ^*w =s(b ) =J(1)$, we obtain
	$\widetilde{J}(\left(
	\begin{array}{ccccc}
	0 & x\\
	y &  0\\
	\end{array}
	\right))=\left(
	\begin{array}{ccccc}
	0 & J(x)\\
	J(y) &  0\\
	\end{array}
	\right).$ 
	As a result,
	\[\left(
	\begin{array}{ccccc}
	J(xy) & 0\\
	0 &  J(xy)\\
	\end{array}
	\right)=\widetilde{J}\Big(\left(
	\begin{array}{ccccc}
	xy & 0\\
	0 &  yx\\
	\end{array}
	\right)\Big)=\Big(\widetilde{J}\left(
	\begin{array}{ccccc}
	0 & x\\
	y &  0\\
	\end{array}
	\right)\Big)^2=\left(
	\begin{array}{ccccc}
	J(x)J(y) & 0\\
	0 &  J(y)J(x)\\
	\end{array}
	\right).\]
	Together with the normality of $J$, we deduce that $J$ is a $*$-homomorphism.
	
	Now we prove  {(iii)}$\Rightarrow${(i)}. Note that if $J:\mathcal M\to\mathcal M$ is a normal $*$-homomorphism, then so is $J_n=I_{M_n}\otimes J:M_n\overline{\otimes}\mathcal M\to M_n\overline{\otimes}\mathcal M$ for all $n\geq 1$, and in particular $J_n$ is a Jordan $*$-homomorphism. In this case $I_{S_p^n}\otimes T:L_p(M_n\overline{\otimes}\mathcal{M})\to L_p(M_n\overline{\otimes}\mathcal{M})$ can be written as $I_{S_p^n}\otimes T=w_nb_nJ_n,$ where $w_n=1_{M_n}\otimes w$ and $b_n=1_{M_n}\otimes b $ with $w$ and $b$ given as in the proof of Theorem \ref{charac}. If $T$ is contractive, it is easy to check that the objects $w_n$, $b_n$ and $J_n$ satisfy the conditions {(i)} to {(iv)} in Theorem \ref{charac} with $C=1$, and by Remark \ref{kotokikorarchilobaki}, $I_{S_p^n}\otimes T$ is also Lamperti and contractive. This completes the proof.
\end{proof}

Based on the previous characterizations, we also provide the following properties of completely Lamperti operators.

\begin{prop}\label{prorule}Let $1\leq p<\infty$ and $T:L_p(\mathcal M)\to L_p(\mathcal M)$ be a completely Lamperti operator. Then for all $x,y\in L_p(\mathcal M)$ with $x^*y=xy^*=0,$ we have $(Tx)^*Ty=Tx(Ty)^*=0.$
\end{prop}
\begin{proof}Note that $x^*y=xy^*=0$ implies that $|x|^2|y|^2=|y|^2|x|^2=0.$ This implies $|x||y|=|y||x|=0.$ Let $w,b,J$ be as in Theorem \ref{charac}. Define $S(x)= bJ(x),$ $x\in \mathcal{S}(\mathcal M)$. Clearly, $S$ extends to a positive   completely Lamperti operator. By Theorem \ref{scharac}, $J$ is a normal $*$-homomorphism. Thus $|Tx|=S(|x|)$ for all $x\in\mathcal{S}(\mathcal M)$. Note that the map $x\mapsto |x|$ is continuous with respect to the $\|\ \|_p$-norm (see e.g. \cite[Theorem 4.4]{kosaki84convex}). Hence by an approximation argument we also have $|Tx|=S(|x|)$ for any $x\in L_p(\mathcal M) $. By Proposition \ref{NOTS}
	we have $S(|x|)S(|y|)=0.$ Therefore, $|Tx||Ty|=0.$ Now multiplying the partial isometry $w$ in the polar decomposition of $Tx$ from the left we obtain $Tx|Ty|=0.$ Taking adjoint and applying the same trick again we obtain $Ty(Tx)^*=0.$ By a similar way we obtain $(Tx)^*Ty=0.$ This completes the proof of the proposition.
\end{proof}
The following proposition shows that compositions of completely Lamperti operators are again completely Lamperti.
\begin{prop}\label{hotatuse}
	Let $1\leq p<\infty.$ Let $T_i:L_p(\mathcal M)\to L_p(\mathcal M), i=1,2$ be two completely Lamperti operators. Then  $T_1T_2:L_p(\mathcal M)\to L_p(\mathcal M)$ is also completely Lamperti.
\end{prop}
\begin{proof}By replacing $T_i$ by $I_{S_p^n} \otimes T_i$ without loss of generality, it suffices to show that $T_1 T_2$ is Lamperti. Let $x,y\in L _p(\mathcal M)$ with $x^*y=xy^*=0.$ Then by Proposition \ref{prorule} we have \[(T_2x)^*T_2y=T_2x(T_2y)^*=0.\] Since $T_1$ is completely Lamperti we have by Proposition \ref{prorule} again \[(T_1T_2x)^*T_1T_2y=T_1T_2x(T_1T_2y)^*=0.\] Therefore, $T_1T_2$ is again Lamperti. This completes the proof.
\end{proof}
\begin{rem}\label{arokoto}
	We will keep in mind throughout the paper the following particular cases of Lamperti and completely Lamperti operators.
	\begin{enumerate}
		\item For $1\leq p\neq 2<\infty,$ any isometry (resp. complete isometry) $T:L_p(\mathcal{M})\to L_p(\mathcal{M})$ is Lamperti (resp. completely Lamperti). Moreover, if $T$ is positive isometry (resp. positive complete isometry) on $L_2(\mathcal{M}),$ then $T$ is Lamperti (resp. completely Lamperti). Indeed, for $p\neq 2,$ the claim immediately follows from Remark \ref{kotokikorarchilobaki} and Theorem \ref{ISOthm}. For $p=2$ and $T$ a positive isometry, we take two $\tau$-finite projections  $e,f$ with $ef=0.$ Note that as $T$ is an isometry, \[\|Te+Tf\|_2^2=\|e+f\|_2^2,\quad \|Te+iTf\|_2^2=\|e+if\|_2^2.\] Therefore, we obtain $\tau (TeTf)=\tau (ef)=0.$ Thus, $(Te)^{\frac{1}{2}} Tf (Te)^{\frac{1}{2}} =0 $. In other words we have, 
		\[ ( (Tf)^{\frac{1}{2}} (Te)^{\frac{1}{2}} )^* ( (Tf)^{\frac{1}{2}} (Te)^{\frac{1}{2}}) =0,\]
		Thus we obtain
		$ (Tf)^{\frac{1}{2}} (Te)^{\frac{1}{2}} =0 $ and hence
		$ TeTf=0$. 
		\item	Let $1\leq p<\infty.$ Let $(\Omega,\Sigma,\mu)$ be a $\sigma$-finite measure space. For any nonsigular automorphism $\Phi$ of $(\Omega,\Sigma,\mu)$, it is well-known that $\Phi$ extends to a map on the set of all finite-valued measurable functions such that $\Phi(\chi_E)=\chi_{\Phi(E)}$ for $E\in\Sigma$ (see \cite{kan79thesis}).   Any Lamperti operator $T:L_p(\Omega,{\Sigma,\mu})\to L_p(\Omega,{\Sigma,\mu})$ is of the form $T(f)(x)=h(x)(\Phi f)(x)$ for some measurable function $h$ and for some $\Phi$ as described above  (see \cite{kan78erglamperti}). Moreover, it follows from Remark \ref{kotokikorarchilobaki} and Theorem \ref{scharac} that $T$ is indeed completely Lamperti.
	\end{enumerate}
\end{rem}
\begin{rem}\label{isothmm1}
	By the proof of Theorem \ref{charac} and Theorem \ref{scharac}, we see that Theorem \ref{ISOthm} is also true for Lamperti isometries  for $p=2$. In particular, it also holds for positive isometries on $L_2(\mathcal M)$ according to Remark \ref{arokoto} (i). 
\end{rem}

\section{Dilation theorem for the convex hull of Lamperti contractions}\label{DILATION}
In this section, we prove an $N$-dilation theorem for the convex hull of Lamperti contractions (tautologically, contractions that separate supports) for all $N\geq 1.$ 
For notational simplicity, in this and   next sections we will denote by $\mathbb{SS}(L_p(\mathcal{M}))$ the class of all support separating contractions  on $L_p(\mathcal{M})$, and by $\mathbb{CSS}(L_p(\mathcal{M}))$ the class of all completely support separating contractions on $L_p(\mathcal{M})$. Also, let $\mathbb{SS}^+(L_p(\mathcal{M}))$ (resp. $\mathbb{CSS}^+(L_p(\mathcal{M}))$) be the subclass of positive and support separating (resp. positive completely support separating) contractions. Moreover, given a family $S$ of operators on $L_p(\mathcal{M})$, we denote by $\operatorname{conv}(S)$ the usual convex hull of $S$ consisting of all operators of the form
\[\sum_{i=1}^{n} \lambda_i T_i,\quad T_i\in S,\  \sum_{i=1}^{n}\lambda_i=1,\ \lambda_i\in \mathbb R_+,\ n\in\mathbb N.\]
And we denote by $\overline{\operatorname{conv}}^{sot}(S)$ the closure of $\operatorname{conv}(S)$ with respect to the strong operator topology.

Before the proof, we first give the following useful lemma.
\begin{lem}\label{lem:d}
	Let $1\leq p<\infty$ and $T:L_p(\mathcal  M)\to L_p(\mathcal  M)$ be a Lamperti contraction with the representation $T(x)=wbJ(x)$ for $x\in \mathcal{S}(\mathcal  M)$ given in Theorem \ref{charac}. 
	
	\emph{(i)} Let  $e,f$ be the two projections in the center of the von Neumann algebra $\mathcal N$ generated by $J(\mathcal M)$ with $e+f=1_{\mathcal N}$ given by Lemma \ref{decom}, such that    $J(\cdot)e $ is a $*$-homomorphism and $J(\cdot)f$ is a $*$-anti-homomorphism. Then the weights defined by \[\widetilde{\tau}(x)= \tau(b ^pJ (x)),\ \widetilde{\tau}_1 (x)=\tau(b ^pJ (x)e),\ \widetilde{\tau}_2 (x)=\tau(b ^pJ (x)f), \quad x\in\mathcal M_{+}\] 
	are normal and tracial.
	
	\emph{(ii)} We have a positive element $0\leq \rho \leq 1$ with $\rho \in \mathcal{Z}(\mathcal  M)$ and \begin{equation*}\|T(x)\|_p^p=\tau(\rho |x|^p)=\widetilde{\tau}( |x|^p) \end{equation*} for all $x\in \mathcal{S}(\mathcal  M)$.
\end{lem} 
\begin{proof} 
	Notice that all the weights $\widetilde{\tau}, \widetilde{\tau}_1, \widetilde{\tau}_2$ are   normal, as explained previously in the proof of Theorem \ref{charac}. For $x\in\mathcal M$, by the traciality of $\tau$ and the commutativity between $J (\mathcal M)$ and spectral projections of $b$, we have
	\[
	\widetilde{\tau}_1 (x^*x)  = \tau(b ^p    J (x^*)J (x) e )  = \tau(b ^p  J (x) J (x^*) e )  = \tau(b ^p  J (x x^*) e )  = \widetilde{\tau}_1 (xx^*).
	\]
	So $\widetilde{\tau}_1$ is also tracial. Similarly we have the traciality for $\widetilde{\tau}_2$ and hence for $\widetilde{\tau} = \widetilde{\tau}_1 + \widetilde{\tau}_2$. In particular, $ \widetilde{\tau}_2 (|x^*|^p)
	= \widetilde{\tau}_2 (|x|^p)$. Together with \eqref{eq:txp} we see that
	\[\|Tx\|_p^p= \widetilde{\tau}_1 (|x|^p) + \widetilde{\tau}_2 (|x^*|^p)
	=  \widetilde{\tau}_1 (|x|^p) + \widetilde{\tau}_2 (|x|^p) = \widetilde{\tau} (|x|^p).\]
	Also, recall that by Theorem \ref{charac}   we have  
	\[\tau(b^pJ(x))\leq \tau(x),\quad x\in\mathcal{M}_+ .\] Therefore, by the noncommutative Radon-Nikodym theorem \cite[Chap. I, \S6.4, Théorème 3]{dixmier69book}, there exists a positive element $\rho$ in the center of $\mathcal{M}$ such that $0\leq \rho\leq 1$ and $\widetilde{\tau}(x)=\tau(\rho x)$ for all $x\in\mathcal M_{+}.$ The proof is complete.
\end{proof}
Now we give the following simultaneous dilation theorem for support separating contractions.
\begin{prop}\label{fcss}Let $1\leq p<\infty.$ Then the set $\mathbb{SS}(L_p(\mathcal{M}))$ has a   simultaneous dilation, and the set $\mathbb{CSS}(L_p(\mathcal{M}))$ has a complete simultaneous dilation.
\end{prop}
\begin{proof}
	Let $T:L_p(\mathcal  M)\to L_p(\mathcal  M)$ be a Lamperti contraction and let $\rho$ be given as in the previous lemma. Then we have \begin{equation}\label{radon1}\|T(x)\|_p^p-\|x\|_p^p=\widetilde{\tau}(|x|^p)-\tau(|x|^p)=\tau((\rho  -1)|x|^p)\end{equation} for all $x\in\mathcal{S}(\mathcal{M}).$ Define \[S_T:L_p(\mathcal{M} )\to L_p(\mathcal{M} ),\quad S_T(x)=(1-\rho  )^{\frac{1}{p}}x, \quad x\in \mathcal{S}(\mathcal{M}).\] Thus we see from \eqref{radon1} that \begin{equation}\label{iso11}
	\|T(x)\|_p^p+\|S_T(x)\|_p^p=\|x\|_p^p
	\end{equation} for all $x\in L_p(\mathcal{M}).$ Consider the linear map \[U_T:\ell_p(L_p(\mathcal{M}))\to\ell_p(L_p(\mathcal{M}))\] defined as the following \[U_T(x_0,x_1,\dots)=(T(x_0),S_T(x_0),x_1,x_2,\dots).\] By \eqref{iso11} $U_T$ becomes an isometry. We also define the maps \[i:L_p(\mathcal{M})\to\ell_p(L_p(\mathcal{M})),\quad 
	i(x)=(x,0,\dots)\] and  \[j:\ell_p(L_p(\mathcal{M}))\to L_p(\mathcal{M}),\quad j(x_0,x_1\dots)=x_0.\]
	Clearly, $i$ is a complete isometry and $j$ is a complete contraction. Note that if $T=w_Tb_TJ_T$ as in Theorem \ref{charac}
	then $U_T=w_{U_T}b_{U_T}J_{U_T},$ where   \[w_{U_T}\colon=(w_T,s((1-\rho  )^{\frac{1}{p}}),1,\dots),\]   is a partial isometry, $b_{U_T}\colon=(b,(1-\rho  )^{\frac{1}{p}},1,\dots)$ is a self-adjoint positive operator affiliated with the von Neumann algebra $\oplus_{n=0}^\infty \mathcal M$ and \[J_{U_T}(x_0,x_1,x_2,\dots)\coloneqq (J(x_0),x_0s((1-\rho  )^{\frac{1}{p}}),x_1,\dots),\quad x_i\in\mathcal{M},\ i\geq 0\] is a normal Jordan $*$-homomorphism  on $\oplus_{n=0}^\infty \mathcal M.$ Therefore, by Theorem \ref{scharac} if $T$ is completely Lamperti, then $J_T$ and $J_{U_T}$ are multiplicative and ${U_T}$ is a complete isometry by Theorem \ref{ISOthm} and Remark \ref{isothmm1}.
	
	Note that for any Lamperti contractions $T_1,\dots,T_n$ on $L_p(\mathcal{M})$, we have 
	\[T_1\dots T_n=jU_{T_1}\dots U_{T_n}i\] for all $n\geq 0.$ This completes the proof.
\end{proof}
\begin{rem}\label{posiisimp}
	In Proposition \ref{fcss}, if $T$ is positive, then $U_T$ is again positive. Moreover, it is clear that $i$ and $j$ are always completely positive.
\end{rem}
\begin{rem}\label{shitty}
	Notice that in Proposition \ref{fcss} each $U_T$ is actually a Lamperti isometry for all $1\leq p<\infty.$ Moreover it is completely Lamperti if so is $T$.
\end{rem}

We remark that these dilations also allow to improve Theorem \ref{scharac}   for positive Lamperti operators. Some part of the results have been pointed out to us by C\'edric Arhancet. 
\begin{prop}\label{prop:cp}
	Let $1\leq p<\infty.$
	Let $T:L_p(\mathcal{M})\to L_p(\mathcal{M})$ be a positive Lamperti operator. Then the following assertions are equivalent:
	\begin{enumerate}
		\item $T$ is completely Lamperti;
		\item $T$ is completely positive;
		\item $T$ is $2$-positive;
		\item The map $J$ in Theorem \ref{charac} is actually a $*$-homomorphism.
	\end{enumerate}
\end{prop} 
\begin{proof}
	By Theorem \ref{scharac}, it suffices to prove the equivalence between {(ii)}, {(iii)} and {(iv)}. If {(iv)} holds, then $J_{U_T}$ in the proof of Proposition \ref{fcss} is also a $*$-homomorphism. Thus according to Theorem \ref{thm:cedric}, $U_T$ is completely positive, and hence so is $T=jU_T i$. Conversely, if $T$ is $2$-positive, then $U_T$ is is also $2$-positive. Therefore by  Theorem \ref{thm:cedric}, $J_{U_T}$ is multiplicative. In particular so is $J$.
\end{proof}
In the following  we will use some tools from \cite{facklergluck19dilation} to enlarge our class of dilatable operators.
\begin{thm}\label{GDI}Let $1<p<\infty.$ Suppose that $S\subseteq B(L_p(\mathcal M))$ has a  simultaneous (resp. complete simultaneous) dilation. Then each operator $T\in \operatorname{conv}(S)$ has an   $N$-dilation (resp. complete $N$-dilation) for all $N\in\mathbb{N}.$  \end{thm} 
\begin{proof}
	We will use the construction given in \cite[Proof of Theorem 4.1]{facklergluck19dilation}.	We take a tuple of scalars $\lambda\coloneqq (\lambda_1,\ldots,\lambda_n)$ with $\sum_{i=1}^n\lambda_i=1$ and {$\lambda_i\geq 0$} for all $1\leq i\leq n.$ Also take $T=\sum_{i=1}^n\lambda_i T_i $ where $T_i\in S$.  As in \cite{facklergluck19dilation}, without loss of generality, we may assume that each $T_i$ is an isometry as $S$ admits a   simultaneous dilation for $1<p<\infty.$ Let us define the set of tuples \[\mathfrak{I}=\{\underline{i}\coloneqq (i_1,\ldots,i_N):\forall\, 1\leq k \leq N, i_k\in \{1,\ldots,n\}\}.\] 
	Denote  \[ \lambda _{\underline{i}}
	= \prod_{k=1}^N\lambda_{i_k},\quad \underline{i}\in\mathfrak{I}.\] Note that $\sum_{\underline{i}\in\mathfrak{I}} \lambda _{\underline{i} }=1.$ Define $Y=\ell_p^{\#{\mathfrak{I}}}(\ell_p^N(L_p(\mathcal{M}))).$ Endowed with the $\ell_p$-direct sum norm, $Y$ becomes a noncommutative $L_p$-space equipped with a normal faithful semifinite trace. Define $Q:Y\to L_p(\mathcal{M})$ as \[Q((x_{k,\underline{i}})_{k\in\{1,\dots,N\},\underline{i} \in\mathfrak I})=\sum_{\underline{i}\in\mathfrak{I}}(\frac{\lambda_{\underline{i}}}{N})^{\frac{1}{p^\prime}}\sum_{k=1}^Nx_{k,\underline{i}},\] where $\frac{1}{p}+\frac{1}{p^\prime}=1.$ Define $J:L_p(\mathcal{M})\to Y$ as $Jx=(J_{\underline{i}} x)_{\underline{i}},$ where \[J_{\underline{i}} x=(\frac{\lambda_{\underline{i}}}{N})^{\frac{1}{p}}(x,\dots,x).\]
	Obviously $J$ is   completely positive; it is a complete isometry since $\sum_{\underline{i}}\lambda_{\underline{i}}=1$. As in \cite{facklergluck19dilation}, one can use H\"older's inequality to check that $Q$ is completely contractive. Moreover $Q$ is completely positive.

	For each $\underline{i}\in\mathfrak I,$ define the linear map $U_{\underline{i}}:\ell_p^N(L_p(\mathcal{M}))\to \ell_p^N(L_p(\mathcal{M}))$ as \[U_{\underline{i}}((x_k)_{1\leq k\leq N}) = (T_{ i _k}x_{\sigma(k)})_{1\leq k\leq N},\] where $\sigma:\{1,\dots,N\}\to\{1,\dots,N\}$ is the $N$-cycle. Note that the map $(x_k)\mapsto (x_{\sigma(k)})$ is completely isometric and completely positive, and that $T_{ i _k}$ is also isometric.  Let us define the linear map \[U:Y\to Y
	,\quad
	U=\oplus_{\underline{i}\in\mathfrak I}U_{\underline{i}}.\] 
	Then $U$ is isometric, and it is moreover completely isometric if so are $T_{ i _k}$'s.
	The identity \[T^n=QU^nJ\] for $n\in\{0,\dots,N\}$ has been proved in \cite[Proof of Theorem 4.1]{facklergluck19dilation}. This completes the proof of the theorem.
\end{proof}

Together with Proposition \ref{fcss}, we immediately obtain the following result in our particular setting.
\begin{cor}\label{CO4.3}Let $1<p<\infty.$ Each operator $T\in \operatorname{conv} (\mathbb{SS}(L_p(\mathcal{M})))$ has an $N$-dilation for all $N\in\mathbb{N},$ and each $T\in \operatorname{conv}(\mathbb{CSS}(L_p(\mathcal{M})))$ has a complete $N$-dilation for all $N\in\mathbb{N}.$ Moreover, if this operator $T $ is positive, then all the maps $Q,U$ and $J$ as in Definition \ref{mdil} can be taken to be positive.
\end{cor}

\begin{rem}\label{rem:dilation}
	We may also consider   dilations instead of   $N$-dilations in Theorem \ref{GDI}; moreover we may consider dilations for the strong operator closures $\overline{\operatorname{conv}}^{sot}(\mathbb{SS}(L_p(\mathcal{M})))$ and $\overline{\operatorname{conv}}^{sot}(\mathbb{CSS}(L_p(\mathcal{M})))$. To this end we need to allow the appearance of Haagerup's noncommutative $L_p$-spaces instead of the usual tracial $L_p$-spaces $L_p (\mathcal N , \tau_{\mathcal N})$ in Definition \ref{dide1} and \ref{mdil}. It is known from \cite{raynaud02ultra} that the class of all Haagerup $L_p$-spaces (over arbitrary von Neumann algebras) is stable under ultraproducts, which fulfills \cite[Assumption 2.1]{facklergluck19dilation}. Thus by \cite[Theorem 2.9]{facklergluck19dilation}, we can extend Corollary \ref{CO4.3} to obtain dilations and complete dilations. This is out of the scope of the paper, and we will leave the details to the reader and restrict ourselves in the semifinite cases. The above Corollary \ref{CO4.3} for complete $N$-dilations is sufficient for our further purpose.
\end{rem}
\begin{rem}[mixed unitary quantum channels]
	It is indeed natural to consider the above dilation theory for $\operatorname{conv} (\mathbb{SS}(L_p(\mathcal{M}))$ in view of various related works on quantum Birkhoff conjectures. For instance, consider  the family $Aut(B( \mathcal H))$ of all automorphisms of the von Neumann algebra $B(  \mathcal H)$ for a finite dimensional Hilbert space $H$. It is well-known that any $T\in Aut(B(\mathcal  H))$ is of the form $Tx=u^* xu$ for all $x\in B(\mathcal H)$ with a fixed unitary $u\in B( \mathcal H)$, which is in particular a completely positive complete isometry on $S_p (\mathcal H)$ for all $1< p <\infty$, and hence is completely Lamperti. The convex hull of $Aut(B(\mathcal  H))$ can be naturally included into the set of all unital completely positive trace preserving maps on $B(\mathcal  H)$, and the inclusion is strict if $\operatorname{dim} \mathcal H\geq 3$; this is applied in \cite{landaustreater93birkhoff} (also see \cite{mendlwolf09birkhoffpb}) to obtain a negative solution to the quantum Birkhoff conjecture. The operators in this inclusion of $\operatorname{conv}(Aut(B(\mathcal  H)))$ are referred to as \emph{mixed unitary quantum channels} in the quantum information theory (see e.g. \cite{chencao09extreme}). It follows from Corollary \ref{CO4.3} that every mixed unitary quantum channel, realized as an operator on $S_p (\mathcal H)$, has a complete $N$-dilation for any $N\geq 1$ and $1< p<\infty$. Also, the particular case of unital completely positive Schur multipliers is studied in \cite{omearapereira13selfdual} and \cite{haagerupmusat11factorization}. A matrix $m\in M_n$ defines a Schur multiplier   $T_m((a_{i,j})_{1\leq i,j\leq n})\colon=(m_{i j} a_{ij})_{1\leq i,j\leq n}$ for all $n\times n$ matrices $((a_{i j})_{1\leq i,j\leq n})$. It is shown in  \cite{omearapereira13selfdual} that $T_m$ is a mixed unitary quantum channel iff $m$ belongs to the convex hull of rank one positive definite matrices with diagonal entries equal to $1$. Note that if $m$ is such a matrix of rank one, then it is of the form $m=(z_i\bar{z_j})_{i,j=1}^n$ with $|z_i|=1$ and consequently $T_m (x)=uxu^*$ for $x\in M_n$ with $u=\sum_{i=1}^nz_ie_{ii}$; in particular $T_m\in Aut(M_n)$ and it is completely Lamperti and completely isometric on $S_p^n$. These observations recover partially some dilation theorems of \cite{arhancet13matsaev}.
\end{rem}

In the following we give a quick application of the previous results. Let $1<p\neq 2<\infty.$ For any complex polynomial $P(z)= \sum_{k=0}^na_kz^k,$ define \[a_P= (\dots,0,a_0,\dots,a_n,0,\dots)\in\ell_1(\mathbb Z)\] with $a_0$ in the $0$-th position. Define the linear operator $C(a_P):\ell_p(\mathbb Z)\to\ell_p(\mathbb Z)$ as \[C(a_P)(b)= a_P*b,\] for $b\in\ell_p(\mathbb Z).$ Also, recall that a von Neumann algebra is said to have QWEP if it is a quotient of a $C^*$-algebra having weak expectation property  (see \cite{ozawa04qwep} for details).
\begin{cor}\label{cor:matsaev}Let $1<p\neq 2<\infty $ and assume that the von Neumann algebra $\mathcal{M}$ has the \text{QWEP}. Let $T\in\overline{\operatorname{conv}}^{sot}(\mathbb{SS}(L_p(\mathcal{M})))$. Then $T$ satisfies the noncommutative Matsaev's  conjecture, i.e.
	\[\|P(T)\|_{L_p(\mathcal{M})\to L_p(\mathcal{M})}\leq \|C(a_P)\otimes I_{S_p} \|_{\ell_p(\mathbb{Z};S_p)\to \ell_p(\mathbb{Z};S_p)}\] for all complex polynomials $P$. Moreover, if  {$T\in\overline{\operatorname{conv}}^{sot}(\mathbb{CSS}(L_p(\mathcal{M})))$}, then we have \[\|P(T)\|_{cb,L_p(\mathcal{M})\to L_p(\mathcal{M})}\leq \|C(a_P)\otimes I_{S_p} \|_{\ell_p(\mathbb{Z}; S_p)\to \ell_p(\mathbb{Z}; S_p)}\] for all complex polynomials $P$.
\end{cor}
\begin{proof}Note that each $T\in \operatorname{conv}(\mathbb{SS}(L_p(\mathcal{M})))$ admits an $N$-dilation for all $N\geq 1.$ By  \cite[Lemma 13.3.3]{brown2008c}, it is easy to see that the von Neumann algebra $\oplus_{n=1}^\infty\mathcal M$ has again the QWEP. Therefore, by \cite{arhancet13matsaev} we   have \begin{equation}\label{mat1}\|P(T)\|_{L_p(\mathcal{M})\to L_p(\mathcal{M})}\leq \|C(a_P)\otimes I_{S_p} \|_{\ell_p(\mathbb{Z};S_p)\to \ell_p(\mathbb{Z};S_p)}\end{equation} for all complex polynomials $P$. For any $T\in \overline{\operatorname{conv}}^{sot}(\mathbb{SS}(L_p(\mathcal{M})))$ there exists a sequence of operators $T_j\in \operatorname{conv}(\mathbb{SS}(L_p(\mathcal{M})))$ such that $T_j\to T$ in strong operator topology. Therefore, for all $x\in L_p(\mathcal{M})$, we have \begin{equation}\label{mat2}\|P(T)x\|_{L_p(\mathcal{M})}\leq\lim_{j\to\infty}\|P(T_j)x-P(T)x\|_{L_p(\mathcal{M})}+\limsup_{j\to\infty}\|P(T_j)x\|_{L_p(\mathcal{M})}.\end{equation}
	The required conclusion follows from \eqref{mat1} and \eqref{mat2}.  {The remaining part of the proof for $T\in\overline{\operatorname{conv}}^{sot}(\mathbb{CSS}(L_p(\mathcal{M})))$ is similar.}
\end{proof}
\par

\section{Ergodic theorems for the convex hull of Lamperti contractions}\label{isodee}
In this section, we prove the maximal ergodic inequality for operators in the closed convex hull of positive Lamperti contractions, or more precisely in the class $\overline{\operatorname{conv}}^{sot}(\mathbb{SS}^+(L_p(\mathcal{M})))$. Based on the dilation theorem established in the previous section, we first need  a maximal ergodic inequality for positive isometries. Recall that throughout the paper  $C_p$ always denotes the best constant of Junge-Xu's maximal ergodic inequality \cite[Theorem 0.1]{jungexu07erg}, which is a fixed distinguished constant depending only on $p$.
\begin{thm}\label{DEEISO}
	Let $1<p<\infty.$ Let $T:L_p(\mathcal M)\to L_p{(\mathcal M)}$ be a positive   isometry. Then \[\Big\|\Big(\frac{1}{ n+1}\sum_{k=0}^nT^k x\Big)_{n\geq 0}\Big)\|_{L_p(\mathcal M;\ell_\infty)} \leq C_p\|x\|_p,\quad \forall\, x\in L_p(\mathcal M).\]
\end{thm}
We will first consider the following auxiliary facts which will be essential in our proof of maximal ergodic inequality for positive isometry.
\begin{lem}\label{lem:extension lamperti}
	Let $1\leq p<\infty.$ Let $T:L_{p}(\mathcal M,\tau)\to L_{p}(\mathcal M,\tau)$ be a positive Lamperti contraction.
	Then $T$ extends to a contraction on $L_{p}(\mathcal M;\ell_{\infty})$.
\end{lem}
\begin{proof}
	Let $(x_n)_{n\geq 0}\in L_p(\mathcal M;\ell_\infty).$ Given $\varepsilon>0,$ let us choose a factorization $x_n=a_1 y_n a_2$ such that $\sup_{n}\|y_n\|_\infty\leq 1$ and $\|a_1\|_{2p}^2=\|a_2 \|_{2p}^2\leq\|(x_n)_{n\geq 0}\|_{L_p(\mathcal M;\ell_\infty)}+\varepsilon.$ By density and without loss of generality, we assume that the elements $x_n,y_n,a_1$ and $a_2$ always belong to $\mathcal S (\mathcal M )$.

For simplicity, we first present the proof when $T$ is just a Jordan homomorphism, i.e. $T=J$ in Theorem \ref{charac}.	Let  the projections $e,f$ and the traces  $\widetilde{\tau}, \widetilde{\tau}_1 ,\widetilde{\tau}_2$ be as in Lemma \ref{lem:d}. Note that we have 
	\begin{align*}
J(x_n)& =J(a_1 y_na_2 )e+J(a_1 y_na_2 )f\\
	&=J(a_1 )J(y_n)J(a_2 )e+J(a_2 )J(y_n)J(a_1 )f\\
	&=(J(a_1 )e+J(a_2 )f)J(y_n)(J(a_2 )e+J(a_1 )f).
	\end{align*}
	We write therefore
	\[J(x_n)=\tilde{a}_1 \tilde{y}_n\tilde{a}_2 \] with $\tilde{a}_1  =J(a_1 )e+J(a_2 )f,$ $\tilde{y}_n =J(y_n)$ and $\tilde{a }_2 =J(a_2 )e+J(a_1 )f$. 
Note that
	\begin{align*} 
	\|\tilde{a }_1\|_{2p}^{2p}&=\tau\Big(\Big((J(a_1 ^*)e+J(a_2 ^*)f)(J(a_1 )e+J(a_2 )f)\Big)^p\Big)\\ 
	&=\tau\Big(\Big((J(|a_1 |^2)e+J(|a_2 ^*|^2)f)\Big)^p\Big) =\tau\Big(J(|a_1 |^{2p})e+J(|a_2 ^*|^{2p})f\Big)\\ 
	&=\widetilde{\tau}_1(|a_1 |^{2p})+\widetilde{\tau}_2(|a_2 ^*|^{2p}),
	\end{align*}
	and similarly,
	\begin{equation*}
	\|\tilde{a }_2\|_{2p}^{2p}=\widetilde{\tau}_1(|a_2 |^{2p})+\widetilde{\tau}_2(|a_1 ^*|^{2p}).
	\end{equation*}
	Thus we have 
	\begin{align*}
	\|\tilde{a }_1\|_{2p}^{2p}\|\tilde{a }_2\|_{2p}^{2p}&=\Big(\widetilde{\tau}_1(|a_1 |^{2p})+\widetilde{\tau}_2(|a_2 ^*|^{2p})\Big)\Big(\widetilde{\tau}_1(|a_2 |^{2p})+\widetilde{\tau}_2(|a_1 ^*|^{2p})\Big)\\\label{newtrality}
	&=\Big(\widetilde{\tau}_1(|a_1 |^{2p})+\widetilde{\tau}_2(|a_2 |^{2p})\Big)\Big(\widetilde{\tau}_1(|a_2 |^{2p})+\widetilde{\tau}_2(|a_1 |^{2p})\Big),
	\end{align*}
	where the last equality follows from the traciality and normality of  $\widetilde{\tau}_1$ and $\widetilde{\tau}_2.$ Recall that we have taken $\|a_1\|_{2p} =\|a_2 \|_{2p} $. Together with Lemma \ref{lem:d} we have further
	\begin{align*}
	\|\tilde{a }_1\|_{2p}^{2p}\|\tilde{a }_2\|_{2p}^{2p}&=\Big(\widetilde{\tau}_1(|a_1 |^{2p})+\|J(|a_2 |^2)\|_p^p-\widetilde{\tau}_1(|a_2 |^{2p})\Big)\Big(\widetilde{\tau}_1(|a_2 |^{2p})+\|J(|a_1 |^2)\|_p^p-\widetilde{\tau}_1(|a_1 |^{2p})\Big)\\
	&\leq \Big(\widetilde{\tau}_1(|a_1 |^{2p})+\|a_2 \|_{2p}^{2p}-\widetilde{\tau}_1(|a_2 |^{2p})\Big)\Big(\|a_1 \|_{2p}^{2p}-\widetilde{\tau}_1(|a_1 |^{2p})+\widetilde{\tau}_1(|a_2 |^{2p})\Big)\\
	&=\Big(\|a_1 \|_{2p}^{2p}+(\widetilde{\tau}_1(|a_1 |^{2p})-\widetilde{\tau}_1(|a_2 |^{2p}))\Big)\Big(\|a_1 \|_{2p}^{2p}-(\widetilde{\tau}_1(|a_1 |^{2p})-\widetilde{\tau}_1(|a_2 |^{2p}))\Big)\\
	&=\|a_1 \|_{2p}^{4p}-(\widetilde{\tau}_1(|a_1 |^{2p})-\widetilde{\tau}_1(|a_2 |^{2p}))^2\\
	&\leq \|a_1 \|_{2p}^{4p}\leq (\|(x_n)_{n\geq 0}\|_{L_p(\mathcal M;\ell_\infty)}+\varepsilon)^{2p}.
	\end{align*}
	In other words $\|\tilde{a }_1\|_{2p}\|\tilde{a }_2\|_{2p}\leq \|(x_n)_{n\geq 0}\|_{L_p(\mathcal M;\ell_\infty)}+\varepsilon.$ Clearly $\sup_{n\geq 0}\|J(y_n)\|_\infty\leq 1.$ 
	This proves that $\|(J(x_n))_{n\geq 0}\|_{L_p(\mathcal M;\ell_\infty)}\leq  \|(x_n)_{n\geq 0}\|_{L_p(\mathcal M;\ell_\infty)}+\varepsilon$ for arbitrary $\varepsilon>0$. In particular $J$ extends to a contraction on $L_p(\mathcal M;\ell_\infty)$.

Now we prove the general case. Let $b,J$ as be given in Theorem \ref{charac} and let the projections $e,f$ and the traces  $\widetilde{\tau}, \widetilde{\tau}_1 ,\widetilde{\tau}_2$ be again as in Lemma \ref{lem:d}. Notice that
	\begin{align*}
	T(x_n)=b(J(a_1 )e+J(a_2 )f)J(y_n)(J(a_2 )e+J(a_1 )f).
	\end{align*}
Hence we have
	$T(x_n)=\tilde{a}_1 \tilde{y}_n\tilde{a}_2 $ with $\tilde{a}_1  =b^{\frac{1}{2}}(J(a_1 )e+J(a_2 )f),$ $\tilde{y}_n =J(y_n)$ and $\tilde{a }_2 =b^{\frac{1}{2}}(J(a_2 )e+J(a_1 )f)$. 
Calculating as before, we obtain
	\begin{align*} 
	\|\tilde{a }_1\|_{2p}^{2p}&=\tau\Big(\Big((J(a_1 ^*)e+J(a_2 ^*)f)b(J(a_1 )e+J(a_2 )f)\Big)^p\Big)\\ 
	&=\tau\Big(\Big(b(J(|a_1 |^2)e+J(|a_2 ^*|^2)f)\Big)^p\Big) =\tau\Big(b^p\Big(J(|a_1 |^{2p})e+J(|a_2 ^*|^{2p})f\Big)\Big)\\ 
	&=\widetilde{\tau}_1(|a_1 |^{2p})+\widetilde{\tau}_2(|a_2 ^*|^{2p}),
	\end{align*}
	and similarly,
$\|\tilde{a }_2\|_{2p}^{2p}=\widetilde{\tau}_1(|a_2 |^{2p})+\widetilde{\tau}_2(|a_1 ^*|^{2p}).$ The rest of the proof is the same as the we did when $T$ was just a Jordan homomorphism.
\end{proof}
\begin{prop}\label{prop:extension}
	Let $1<p<\infty.$ Let $T:L_{p}(\mathcal M,\tau)\to L_{p}(\mathcal M,\tau)$ be a positive   isometry.
	Then $T$ extends to an isometry on $L_{p}(\mathcal M;\ell_{\infty})$.\end{prop}
\begin{proof}
For simplicity, first we give the proof of the proposition when $T=J$,
	where $J:\mathcal M\to \mathcal M$ is a normal Jordan monomorphism. Denote by $\mathcal N$ the von Neumann algebra generated by $J(\mathcal M)$. By Lemma \ref{decom}, we may write $\mathcal N = \mathcal N _1 \oplus \mathcal N_2$ where $\mathcal N_1 $ and $\mathcal N_2$ are two von Neumann subalgebras of $\mathcal N$, and write $J=J_1 + J_2$ such that $J_1 : \mathcal M \to \mathcal N_1$ is a normal $*$-homomorphism  and $J_2:\mathcal M \to \mathcal N_2$ is a normal $*$-anti-homomorphism. Let $\sigma: \mathcal N_2 \to \mathcal N_2^{op}$ be the usual opposite map and define 
	\[\Sigma: \mathcal N \to  \mathcal N _1 \oplus \mathcal N_2^{op},\quad \Sigma=\operatorname{Id}_{\mathcal N_1} \oplus  \sigma.\] Then $\Sigma\circ J$ is a normal $*$-homomorphism and in particular its image $\Sigma(J(\mathcal M))$ is a von Neumann subalgebra of $\mathcal N _1 \oplus \mathcal N_2^{op}$. We consider the faithful weight
	\[
	\varphi:\Sigma(J(\mathcal M))_{+}\to[0,\infty],\quad x\mapsto\tau(\Sigma^{-1} x).
	\]
	We claim that $\varphi$ is a normal semifinite trace on $\Sigma(J(\mathcal M))$. Indeed, for $x\in \mathcal M$, we have
	\begin{align*}
	\varphi ((\Sigma Jx^*)(\Sigma Jx)) & = \varphi((J_1 x^*)(J_1 x)) + \varphi((\sigma J_2 x^*)(\sigma J_2 x)) \\
	& = \varphi(J_1 ( x^* x )) + \varphi(\sigma((  J_2 x)( J_2 x^*)) )
	= \varphi(J_1 (x^*  x)) + \varphi(\sigma( J_2  (x^*x))) \\
	& =\tau (J_1(x^*x)) + \tau (J_2(x^*x)). 
	\end{align*}
	Thus by Lemma \ref{lem:d} we see that  $\varphi$ is tracial. We consider
	the associated noncommutative $L_{p}$-space $L_{p}(\Sigma(J(\mathcal M)),\varphi)$.
	It is easy to see that $\tilde{J}:=\Sigma\circ J:L_{p}(\mathcal M,\tau)\to L_{p}(\Sigma(J(\mathcal M)),\varphi)$ extends to a positive surjective isometry.
	
	As a result we see that $\tilde{J}^{-1}$ is well-defined, positive and isometric on $ L_{p}(\Sigma(J(\mathcal M)),\varphi) $. Therefore, for any positive sequence $(x_n)_{n\geq 0}\subset  L_p (\mathcal M ) $ and any $a \in  L_{p}(\Sigma(J(\mathcal M)),\varphi) _+$, we see that $\tilde{J} x_n \leq a$  if and only if $ x_n \leq \tilde{J}^{-1}a$. Recall that 
	\[\|(x_n )_{n\geq 0}\|_{L_{p}(\mathcal M,\tau;\ell_\infty)}=\inf\{\|a\|_p:x_n\leq a, a\in L_p(\mathcal M,\tau)_{+}\}.\]
	We see that $\tilde{J}$ extends to an isometry from $L_{p}(\mathcal M,\tau;\ell_\infty)$ onto $L_{p}(\Sigma(J(\mathcal M)),\varphi;\ell_\infty)$. 
	
	It remains to prove  that the embedding \[L_{p}(\Sigma(J(\mathcal M)),\varphi;\ell_{\infty})\to L_{p}(\mathcal M,\tau;\ell_{\infty}),\quad (x_{n})_{n\geq 0}\mapsto(\Sigma^{-1} x_{n})_{n\geq 0}\]
	is isometric. 
	Let $1<p^\prime<\infty$ with $\frac{1}{p}+\frac{1}{p^\prime}=1$. For $y\in \Sigma(J(\mathcal M))_{+}$,
	we have the equality
	$
	\| \Sigma^{-1}y\|_{L_{p^\prime}(\mathcal M,\tau)}^{p^\prime}=\|y\|^{p^\prime}_{L_{p^\prime}(\Sigma(J(\mathcal M),\varphi)}.$
	So the map
	\[
	\iota:L_{p^\prime}(\Sigma(J(\mathcal M)),\varphi;\ell_{1})\to L_{p^\prime}(\mathcal M,\tau;\ell_{1}),\quad(y_{n})_{n\geq 0}\mapsto( \Sigma^{-1}y_{n})_{n\geq 0}
	\]
	is isometric. Note that for $(x_{n})_{n\geq 0}\in L_{p}(\Sigma(J(\mathcal M)),\varphi;\ell_{\infty}),(y_{n})_{n\geq 0}\in L_{p^\prime}(\Sigma(J(\mathcal M)),\varphi;\ell_{1})$,
	\begin{align*}
	\Big\langle\iota^{*}((\Sigma^{-1}x_{n})_{n\geq 0}),(y_{n})_{n\geq 0}\Big\rangle & =   \Big\langle(\Sigma^{-1}x_{n})_{n\geq 0},\iota((y_{n})_{n\geq 0})\Big\rangle=\sum_{n\geq 0}\tau((\Sigma^{-1}x_{n})(\Sigma^{-1}y_{n})).
	\end{align*}
	We write $x_n = (x_n^{(1)},x_n^{(2 )})\in L_p(\mathcal N_1 )\oplus L_p(\mathcal N_2^{op} )$ and $y_n = (y_n^{(1)},y_n^{(2)})\in L_{p^\prime}(\mathcal N_1 )\oplus L_{p^\prime}(\mathcal N_2^{op} )$. Then by the traciality of $\tau$ we obtain
	\begin{align*}
	\tau((\Sigma^{-1}x_{n})(\Sigma^{-1}y_{n}))
	&=\tau(x_{n}^{(1)}  y_{n}^{(1)} ) + \tau((\sigma^{-1}x_{n}^{(2)})(\sigma^{-1}y_{n}^{(2)})) \\
	&= \tau(x_{n}^{(1)}  y_{n}^{(1)} )+\tau((\sigma^{-1}y_{n}^{(2)})(\sigma^{-1}x_{n}^{(2)}))\\
	&= \tau(x_{n}^{(1)}  y_{n}^{(1)} ) + \tau(\sigma^{-1}( x_{n}^{(2)} y_{n}^{(2)}))\\
	&= \tau(\Sigma^{-1}x_{n} y_{n}) = \varphi(x_{n}y_{n}).
	\end{align*}
	Thus combined with the previous equalities we obtain 
	\[ \Big\langle\iota^{*}((\Sigma^{-1}x_{n})_{n\geq 0}),(y_{n})_{n\geq 0}\Big\rangle  =   \sum_{n\geq 0}\varphi(x_{n}y_{n})=\Big\langle(x_{n})_{n\geq 0},(y_{n})_{n\geq 0}\Big\rangle.\]
	Therefore, we have $\iota^{*}((\Sigma^{-1}x_{n})_{n\geq 0})=(x_{n})_{n\geq 0}$. Recall that $J$ always extends to a contraction on $L_{p}(\mathcal M,\tau;\ell_{\infty})$ by Lemma \ref{lem:extension lamperti}. Hence, we observe that
	\begin{align*}
	\|(x_{n})_{n\geq 0}\|_{L_{p}(\Sigma(J(\mathcal M)),\varphi;\ell_{\infty})} & =   \|\iota^{*}((\Sigma^{-1}x_{n})_{n\geq 0})\|_{L_{p}(\Sigma(J(\mathcal M)),\varphi;\ell_{\infty})}\leq\|(\Sigma^{-1}x_{n})_{n\geq 0}\|_{L_{p}(\mathcal M,\tau;\ell_{\infty})}\\
	& =   \|(J\tilde{J}^{-1}x_{n})_{n\geq 0}\|_{L_{p}(\mathcal M,\tau;\ell_{\infty})}\leq\|(\tilde{J}^{-1}x_{n})_{n\geq 0}\|_{L_{p}(\mathcal M,\tau;\ell_{\infty})}\\
	& =   \|\tilde{J}((\tilde{J}^{-1}x_{n})_{n\geq 0})\|_{L_{p}(\Sigma(J(\mathcal M)),\varphi;\ell_{\infty})}=\|(x_{n})_{n\geq 0}\|_{L_{p}(\Sigma(J(\mathcal M)),\varphi;\ell_{\infty})}.
	\end{align*}
	Therefore $\|(x_{n})_{n\geq 0}\|_{L_{p}(\Sigma(J(\mathcal M)),\varphi;\ell_{\infty})}=\|(\Sigma^{-1}x_{n})_{n\geq 0}\|_{L_{p}(\mathcal M,\tau;\ell_{\infty})}$,  
	as desired.

	Now let us sketch the proof for the general case. By Theorem \ref{ISOthm},  Theorem \ref{charac} and Remark \ref{isothmm1}  we have $T=bJ$,
	where $J$ and $b$ are as in Theorem \ref{ISOthm}. Taking $\mathcal{N},$ $\mathcal{N}_1,$ $\mathcal{N}_2,$ $J_1$, $J_2$,  $\sigma$ and $\Sigma$ as in the beginning of the proof, we consider the  faithful weight 
	\[
	\varphi:\Sigma(J(\mathcal M))_{+}\to[0,\infty],\quad x\mapsto\tau(b^{p}\Sigma^{-1} x).
	\]
	Again $\varphi$ is a normal semifinite trace on $\Sigma(J(\mathcal M))$. This is because for $x\in \mathcal M$, we can argue as before to obtain
	\begin{align*}
	\varphi ((\Sigma Jx^*)(\Sigma Jx)) & 
	= \varphi(J_1 (x^*  x)) + \varphi(\sigma( J_2  (x^*x)))=\tau (b^{p}J_1(x^*x)) + \tau (b^{p}J_2(x^*x)). 
	\end{align*}
	Hence it follows again from Lemma \ref{lem:d} that  $\varphi$ is tracial. One can see that $\Sigma\circ J$ extends to a positive surjective isometry 
	\[
	\tilde{J}:L_{p}(\mathcal M,\tau)\to L_{p}(\Sigma(J(\mathcal M)),\varphi),\quad x\mapsto \Sigma(Jx),
	\]
	since for $x\in \mathcal{S}(\mathcal M)$,  
	\begin{align*}
	\|\tilde{J}x\|_{L_{p}(\Sigma(J(\mathcal M)),\varphi)}^{p}&=\varphi(|\Sigma(Jx)|^{p})=\tau(b^{p}\Sigma^{-1}(|\Sigma(Jx)|^{p}))=\tau(b^{p}| Jx|{}^{p})
	=\tau( |b Jx|{}^{p}) \\
	&=\|Tx\|_{L_{p}(\mathcal M,\tau)}^p=\|x\|_{L_{p}(\mathcal M,\tau)}^p.
	\end{align*}
	Therefore, arguing as when $T$ was a Jordan monomorphism,
	we see that $\tilde{J}$ extends to an isometry from $L_{p}(\mathcal M,\tau;\ell_\infty)$ onto $L_{p}(\Sigma(J(\mathcal M)),\varphi;\ell_\infty)$. 
	
	It is now enough to prove that the embedding \[L_{p}(\Sigma(J(\mathcal M)),\varphi;\ell_{\infty})\to L_{p}(\mathcal M,\tau;\ell_{\infty}),\quad (x_{n})_{n\geq 0}\mapsto(b\Sigma^{-1} x_{n})_{n\geq 0}\]
	is an isometry. 
	Let $1<p^\prime<\infty$ with $\frac{1}{p}+\frac{1}{p^\prime}=1$. For $y\in \Sigma(J(\mathcal M))_{+}$,
	we obtain
	\[
	\|b^{p/p^\prime} \Sigma^{-1}y\|_{L_{p^\prime}(\mathcal M,\tau)}^{p^\prime}=\tau(b^{p}\Sigma^{-1}(y^{p^\prime}))=\varphi(y^{p^\prime})=\|y\|^{p^\prime}_{L_{p^\prime}(\Sigma(J(\mathcal M),\varphi)}.
	\]
	Thus the map
	\[
	\iota:L_{p^\prime}(\Sigma(J(\mathcal M)),\varphi;\ell_{1})\to L_{p^\prime}(\mathcal M,\tau;\ell_{1}),\quad(y_{n})_{n\geq 0}\mapsto(b^{p/p^\prime} \Sigma^{-1}y_{n})_{n\geq 0}
	\]
	is an isometry. We have for $(x_{n})_{n\geq 0}\in L_{p}(\Sigma(J(\mathcal M)),\varphi;\ell_{\infty}),(y_{n})_{n\geq 0}\in L_{p^\prime}(\Sigma(J(\mathcal M)),\varphi;\ell_{1})$,
	\begin{align*}
	\Big\langle\iota^{*}((b\Sigma^{-1}x_{n})_{n\geq 0}),(y_{n})_{n\geq 0}\Big\rangle & = \sum_{n\geq 0}\tau(b(\Sigma^{-1}x_{n})b^{p/p^\prime}(\Sigma^{-1}y_{n}))=\sum_{n\geq 0}\tau(b^{p}(\Sigma^{-1}x_{n})(\Sigma^{-1}y_{n})).
	\end{align*}
	For $x_n\in L_p(\mathcal N_1 )\oplus L_p(\mathcal N_2^{op} )$ and $y_n\in L_{p^\prime}(\mathcal N_1 )\oplus L_{p^\prime}(\mathcal N_2^{op} )$, using the traciality of $\tau$ and the property of $b$, we can obtain by arguing as before
	\begin{align*}
	\tau(b^{p}(\Sigma^{-1}x_{n})(\Sigma^{-1}y_{n}))= \tau(b^{p}(\Sigma^{-1}x_{n} y_{n})) = \varphi(x_{n}y_{n}).
	\end{align*}
	As before we have 
	\[ \Big\langle\iota^{*}((b\Sigma^{-1}x_{n})_{n\geq 0}),(y_{n})_{n\geq 0}\Big\rangle  =   \sum_{n\geq 0}\varphi(x_{n}y_{n})=\Big\langle(x_{n})_{n\geq 0},(y_{n})_{n\geq 0}\Big\rangle.\]
	Therefore, we have $\iota^{*}((b\Sigma^{-1}x_{n})_{n\geq 0})=(x_{n})_{n\geq 0}$.  Hence, by Lemma \ref{lem:extension lamperti} we see that by arguing as the case when $T$ was just a Jordan monomorphism,
	\begin{align*}
	\|(x_{n})_{n\geq 0}\|_{L_{p}(\Sigma(J(\mathcal M)),\varphi;\ell_{\infty})} & =   \|\iota^{*}((b\Sigma^{-1}x_{n})_{n\geq 0})\|_{L_{p}(\Sigma(J(\mathcal M)),\varphi;\ell_{\infty})}\leq\|(b\Sigma^{-1}x_{n})_{n\geq 0}\|_{L_{p}(\mathcal M,\tau;\ell_{\infty})}\\
	& =   \|(T\tilde{J}^{-1}x_{n})_{n\geq 0}\|_{L_{p}(\mathcal M,\tau;\ell_{\infty})}\leq\|(x_{n})_{n\geq 0}\|_{L_{p}(\Sigma(J(\mathcal M)),\varphi;\ell_{\infty})}.
	\end{align*}
	Therefore we have $\|(x_{n})_{n\geq 0}\|_{L_{p}(\Sigma(J(\mathcal M)),\varphi;\ell_{\infty})}=\|(b\Sigma^{-1}x_{n})_{n\geq 0}\|_{L_{p}(\mathcal M,\tau;\ell_{\infty})}$. This completes the proof of the proposition.
\end{proof}
Now Theorem \ref{DEEISO} follows from the noncommutative transference principle adapted from \cite[Theorem 3.1]{hongliaowang15erg} and Junge-Xu's maximal ergodic inequality \cite{jungexu07erg}. We provide details for the convenience of the readers.
\begin{proof}[Proof of Theorem \ref{DEEISO}]
	In this proof we fix an arbitrary positive integer $N\geq 1$. We write 
	$A_n   = \frac{1}{ n+1}\sum_{k=0}^nT^k$ and
	\[ A_{n}' : L_p ( \mathbb N_0 ; L_p(\mathcal{\mathcal{M}})) \to  L_p ( \mathbb N_0 ; L_p(\mathcal{\mathcal{M}})), \quad
	A_n' f (k)= \frac{1}{n+1} \sum_{l=0}^{n} f (l+k),\quad\forall k\in\mathbb{N}_0 . \]
	We consider $(A_{n}'f)_{0\leq n\leq N}\in L_{p}(\ell_{\infty}(\mathbb N_0)\overline{\otimes}\mathcal{\mathcal{M}};\ell_{\infty})$,
	and for any $\varepsilon>0$ we take a factorization $A_{n}'f=aF_{n}b$
	such that $a,b\in L_{2p}(\ell_{\infty}(\mathbb N_0)\overline{\otimes}\mathcal{\mathcal{M}})$,
	$F_{n}\in \ell_{\infty}(\mathbb N_0)\overline{\otimes}\mathcal{\mathcal{M}}$ and
	\[
	\|a\|_{2p}\sup_{0\leq n\leq N}\|F_{n}\|_{\infty}\|b\|_{2p}\leq\Big\|(A_{n}'f)_{0\leq n\leq N}\Big\|_{L_{p}(\ell_{\infty}(\mathbb N_0)\overline{\otimes}\mathcal{\mathcal{M}};\ell_{\infty})}+\varepsilon.
	\]
	Then we have
	\begin{align*}
	\sum_{k\geq 0}\Big\|\Big(A_{n}'f(k)\Big)_{0\leq n\leq N}\Big\|_{L_{p}({\mathcal{M}};\ell_{\infty})}^{p}  & \leq \sum_{k\geq 0} \|a(k)\|_{2p}^{p}\sup_{0\leq n\leq N}\|F_{n}(k)\|_{\infty}^{p}\|b(k)\|_{2p}^{p} \\
	& \leq\|a\|_{2p}^{p}\sup_{0\leq n\leq N}\|F_{n}\|_{\infty}^{p}\|b\|_{2p}^{p}\leq\left(\Big\|(A_{n}'f)_{0\leq n\leq N}\Big\|_{L_{p}(\ell_{\infty}(\mathbb N_0)\bar{\otimes}\mathcal{\mathcal{M}};\ell_{\infty})}+\varepsilon\right)^{p}.
	\end{align*}
	Since $\varepsilon$ is arbitrarily chosen, we obtain
	\begin{equation}\label{eq:transf}
	\sum_{k\geq 0}\Big\|\Big(A_{n}'f(k)\Big)_{0\leq n\leq N}\Big\|_{L_{p}({\mathcal{M}};\ell_{\infty})}^{p}\leq\Big\|\Big(A_{n}'f\Big)_{0\leq n\leq N}\Big\|_{L_{p}(\ell_{\infty}(\mathbb N_0)\overline{\otimes}\mathcal{\mathcal{M}};\ell_{\infty})}^{p}.
	\end{equation}
	Fix  $x\in L_p (\mathcal M) $. We define a $L_{p}(\mathcal{\mathcal{M}})$-valued function $f_m$ on $\mathbb N_0$ 
	as 
	\[
	f_m(l)= T^l x ,\quad \text{if }l\leq m+N ; \quad f_m(l)= 0  \, \text{ otherwise}.
	\]
	Then for all $0\leq k\leq m$ and $0\leq n\leq N$,
	\begin{equation*}
	T^k A_n x= \frac{1}{n+1} \sum_{l=0}^{n}T^{k+l}x = \frac{1}{n+1} \sum_{l=0}^{n} f_m (l+k) = A_n ' f_m (k).\label{eq:transf alpha_g An}
	\end{equation*}
	Note that the previous proposition yields that for all $0 \leq k \leq m$, we have 
	\[\Big\|\Big( A_n x\Big)_{0\leq n\leq N}\Big\|_{L_{p}({\mathcal{M}};\ell_{\infty})} = \Big\|\Big( T^k A_n x\Big)_{0\leq n\leq N}\Big\|_{L_{p}({\mathcal{M}};\ell_{\infty})}=\Big\|\Big( A_n ' f_m (k)\Big)_{0\leq n\leq N}\Big\|_{L_{p}({\mathcal{M}};\ell_{\infty})},\]
	and hence for any $m\geq 0$,
	\[\Big\|\Big( A_n x\Big)_{0\leq n\leq N}\Big\|_{L_{p}({\mathcal{M}};\ell_{\infty})}^p = \frac{1}{m+1}\sum_{k=0}^{m} \Big\|\Big( A_n ' f_m (k)\Big)_{0\leq n\leq N}\Big\|_{L_{p}({\mathcal{M}};\ell_{\infty})}^p
	\leq\frac{1}{m+1}\Big\|\Big(A_{n}'f_m\Big)_{0\leq n\leq N}\Big\|_{L_{p}(\ell_{\infty}(\mathbb N_0)\overline{\otimes}\mathcal{\mathcal{M}};\ell_{\infty})}^{p} .\]
	Recall that by \cite{jungexu07erg},
	\[ \Big\|\Big(A_{n}'f_m\Big)_{0\leq n\leq N}\Big\|_{L_{p}(\ell_{\infty}(\mathbb N_0)\overline{\otimes}\mathcal{\mathcal{M}};\ell_{\infty})} \leq C_p \|f_m \|_p\] for a constant $C_p$ depending only on $p$ since $f\mapsto f(\cdot +1)$ is a Dunford-Schwartz operator on $\ell_{\infty}(\mathbb N_0)\overline{\otimes}\mathcal{\mathcal{M}}$.
	Thus together with \eqref{eq:transf} we see that
	\begin{equation*}
	\Big\|\Big( A_n x\Big)_{0\leq n\leq N}\Big\|_{L_{p}({\mathcal{M}};\ell_{\infty})}^p \leq\frac{C_p ^p}{m+1}\|f_m\|_{p}^{p}
	=\frac{C_p ^p}{m+1} \sum_{l=0}^{m+N}\|f_m (l)\|_{p}^{p} 
	=\frac{C_p ^p}{m+1} \sum_{l=0}^{m+N}\|T^l x\|_{p}^{p}
	=\frac{C_p ^p (m+N+1)}{m+1}  \| x\|_{p}^{p}.
	\end{equation*}
	Since $m$ is arbitrarily chosen,  
	we get
	\[
	\Big\|\Big( A_n x\Big)_{0\leq n\leq N}\Big\|_{L_{p}({\mathcal{M}};\ell_{\infty})} \leq C_p \|x\|_{p}.
	\]
	This completes the proof of the theorem by using Proposition \ref{FINT}.
\end{proof}

Based on the maximal ergodic theorem for isometries and the dilation theorem, now we can conclude the proof of Theorem \ref{main}, that is, the maximal ergodic theorem for contractions in $\overline{\operatorname{conv}}^{sot}(\mathbb{SS}^+(L_p(\mathcal{M})))$.

\begin{proof}[Proof of Theorem \ref{main}]We write  
	$A_n  (T) = \frac{1}{ n+1}\sum_{k=0}^nT^k$. Fix an arbitrary $N\geq 1 $. Take a sequence  $(T_j)\subseteq \operatorname{conv} (\mathbb{SS}^+(L_p(\mathcal{M})))$ so that $T_j$ converges to $T$ strongly. By Corollary \ref{CO4.3}, there exist positive contractions $Q_{N,j},J_{N,j}$ and positive isometries $U_{N,j}$ such that we have  $T_j^n=Q_{N,j}U_{N,j}^nJ_{N,j}$ for all $0\leq n\leq N.$ Therefore, as each $U_{N,j}$ admits a maximal ergodic inequality with constant $C_p$ by Theorem \ref{DEEISO} and $Q_{N,j},j_{N,j}$ extend to contractions on $L_p(\mathcal M;\ell_\infty^{N+1})$ (see e.g. \cite{jungexu07erg}), we have
	\[\|(A_n(T_j)x)_{n=0}^N\|_{L_p(\mathcal M;\ell_\infty^{N+1})} \leq \|(A_n(U_{N,j})x)_{n=0}^N\|_{L_p(\mathcal M;\ell_\infty^{N+1})}  \leq C_p \|x\|_p,\quad
	x\in L_p(\mathcal M).\]
	Then for any $x\in L_p(\mathcal M),$ and $N\geq 1$ we have
	\begin{align*}
	&\|(A_n(T)x)_{n=0}^N\|_{L_p(\mathcal M;\ell_\infty^{N+1})}\\
	\leq\ & \|(A_n(T_j)x)_{n=0}^N\|_{L_p(\mathcal M;\ell_\infty^{N+1})}+\|(A_n(T)x-A_n(T_j)x)_{n=0}^N\|_{L_p(\mathcal M;\ell_\infty^{N+1})}\\
	\leq\ & C_p\|x\|_{L_p(\mathcal M)}+\sum_{n=0}^N\|A_n(T)x-A_n(T_j)x\|_{L_p(\mathcal M)}.
	\end{align*}
	The result follows by taking $j\to\infty$ and using Proposition \ref{FINT}.	
\end{proof}
As mentioned previously, for $\mathcal M = L_\infty([0,1])$, our result recovers Ackoglu's ergodic theorem. In the following we remark that we may also obtain the general operator-valued version of Ackoglu's theorem.
\begin{cor}\label{cbackoglu}
	Let $1<p<\infty$ and $(\Omega,\mu)$ be a $\sigma$-finite measure space. Then for any positive contraction $T:L_p(\Omega)\to L_p(\Omega)$ and any semifinite von Neumann algebra $\mathcal M$, we have
	\[\Big\|\Big( \frac{1}{ n+1}\sum_{k=0}^n(T \otimes I_{L_p (\mathcal M ) } )^k x\Big)_{n\geq 0} \Big\|_p\leq C_p\|x\|_p,\quad \forall\, x\in L_p(L_\infty(\Omega)\overline{\otimes} \mathcal M).\]
\end{cor}

\begin{rem}
Note that $T \otimes I_{L_p (\mathcal M )}$ is just a completely positive contraction, and no existing theory can be applied immediately to conclude the ergodic maximal inequality. Instead, we exploit Ackoglu's dilation theorem and Theorem \ref{main} for Lamperti operators.
\end{rem}

\begin{proof}
	By Ackoglu's dilation theorem \cite{akcoglu75erg,akcoglusucheston77dilation}, we may write $T^k= Q U^k J$ for all $k\geq 1$, where $J:L_p(\Omega)\to L_p(\Omega')$ and $Q:L_p(\Omega')\to L_p(\Omega)$ are positive contractions, and $U:L_p(\Omega')\to L_p(\Omega')$ is a positive invertible isometry, and $\Omega'$ is a certain measure space. Also, $U$ is positive and Lamperti by Remark \ref{arokoto} and Remark \ref{isothmm1}, and consequently a completely positive and completely Lamperti isometry by Remark \ref{arokoto}. Therefore by Theorem \ref{main}, we have
	\[\Big\|\Big(\frac{1}{ n+1}\sum_{k=0}^n(U \otimes I_{L_p (\mathcal M ) } )^k x \Big)_{n\geq 0}\Big\|_p\leq C_p\|x\|_p,\quad \forall\, x\in L_p(L_\infty(\Omega')\overline{\otimes} \mathcal M).\]
	Note that $T \otimes I_{L_p (\mathcal M ) }$, $J \otimes I_{L_p (\mathcal M ) }$ and $Q \otimes I_{L_p (\mathcal M )} $ are again positive contractions (see for instance \cite[Theorem 2.17 and Proposition 2.21]{arhancetkriegler17decomp} and \cite{junge04fubini}). Thus the proof is complete.
\end{proof}

\section{Ergodic theorem for power bounded doubly Lamperti operators}\label{lamp}

This section is devoted to the proof of our main result, i.e., Theorem \ref{main1}.
Our key ingredient is Theorem \ref{CHARACDL}, which is a technical structural theorem for the \emph{doubly Lamperti} operators (i.e. a Lamperti operator whose adjoint is also Lamperti). The proof is quite lengthy compared to that of the classical one.  We start off with a refined study of the structure of Lamperti operators.

To this end we fix some notation. Let $1\leq p<\infty$ and $T:L_p(\mathcal  M)\to L_p(\mathcal  M)$ be a   positive   Lamperti contraction with the representation $T(x)=bJ(x) $ for $x\in \mathcal{S}(\mathcal  M) $ given in Theorem \ref{charac}. Recall that by Lemma \ref{lem:d} there exists a positive operator $0\leq \rho_T \leq 1$ with $\rho_T \in \mathcal{Z}(\mathcal  M)$ such that \begin{equation}\label{repp}\|T(x)\|_p^p=\tau(\rho_T x^p)=\tau(b^pJ(x^p))\end{equation} for all $x\in \mathcal  M _+$. Denote by $p_0 \in \mathcal{Z}(\mathcal  M)$ the projection onto  $\ker \rho_T $ (in other words we set $p_0 = 1 - s(\rho_T)$)  and write $p_1= p_0^{\perp}=s(\rho_T).$ Also take $\tilde p _0$ to be projection onto  $\ker (1-\rho_T) $ or equivalently $\tilde p _0=1- s(1-\rho_T)$. Throughout the rest of this paper, we maintain the notation introduced here.
\begin{prop}\label{properties}Let $1\leq p<\infty$ and $T:L_p(\mathcal  M)\to L_p(\mathcal  M)$ be a   positive   Lamperti contraction. Then the following statements hold.
	\begin{enumerate}
		\item $T\big{|}_{L_p(p_0\mathcal  M p_0)}=0$ and $T\big{|}_{L_p(\tilde p _0\mathcal  M \tilde p _0)}$ is an isometry;
	
		\item The following statements are equivalent:
		\begin{enumerate}
			\item $T$ is injective;
			\item $p_0=0;$
			\item $J$ is injective.
		\end{enumerate}
		\item Suppose that $T$ is surjective. Then we have
		\begin{enumerate}
			\item $J$ is surjective and $s(b)=J(1)=1,$ moreover $T$ and $J$ are injective on $L_p(p_1 \mathcal M p_1)$ and $p_1 \mathcal M p_1$ respectively;
		\item for some constant $C> 0,$ $p_1\rho_T \geq Cp_1.$
		\end{enumerate}
	\end{enumerate}
\end{prop}
\begin{proof}
	(i) Note that for any $x\in \mathcal{S}(\mathcal  M)_{+},$ by \eqref{repp} we see that \[\|T(p_0xp_0)\|_p^p=\tau(\rho_T p_0(p_0xp_0)^pp_0)=0.\] Therefore, we have $T(p_0xp_0)=0.$ This shows that $L_p(p_0\mathcal  M p_0)\subseteq\ker T.$ 
	
	On the other hand, for any $x\in\mathcal{S}(\mathcal  M),$ we have $(1-\rho_T )\tilde p _0|\tilde p _0 x\tilde p _0|^p =0.$ Therefore, we obtain $\rho_T  |\tilde p _0 x\tilde p _0|^p = |\tilde p _0 x\tilde p _0 |^p .$ By using \eqref{repp}, this shows that $T\big{|}_{L_p(\tilde p _0 \mathcal M \tilde p _0 )}$ is an isometry. 
	
	(ii) By (i), $L_p(p_0\mathcal  M p_0)\subseteq\ker T$,  so it is clear that (a) implies (b).
	
	Recall that   $\tau(b^pJ(x))=\tau(\rho_T x)$ for $x\in \mathcal M_{+}$ by \eqref{repp}. If $J(x)=0$ for some nonzero $x \in\mathcal M $, then $J(|x|)=0$ and hence $\tau(\rho_T |x|)=0$. By the faithfulness of $\tau$ we obtain $\rho_T^{1/2} |x| \rho_T ^{1/2}=0$. Hence $p_1 |x| p_1 =0$, which means that $p_1 \neq 1$ and $p_0\neq 0$. Thus (b) implies (c).
	
	To see that (c) implies (a), we suppose that $T(y)=0$ for some $y\in L_p(\mathcal M).$ By the decomposition $T(y)=bJ(y)$, we see that $T(y^*)=T(y)^*=0.$    Thus $T(\operatorname{Re}\,y)=T(\operatorname{Im}\,y)=0$, where $\operatorname{Re}\,y$ and $\operatorname{Im}\,y$ denote the real and imaginary part of $y$ respectively. By Lemma \ref{NOTS}  we see that $T(|\operatorname{Re}\,y |)=T(|\operatorname{Im}\,y |)=0$. Write $x=|\operatorname{Re}\,y|$ and take a positive sequence $(x_n)_{n\geq 1}\subseteq \mathcal S (\mathcal M)_{+}$ as in Lemma \ref{1.10}. Since $T$ is positive and $x_n\leq x,$ we have $ T(x_n )\leq 0.$ Thus $T (x_n) =0$ for all $n\geq 1.$ Since $s(x_n)\uparrow s(x),$ we have $J(s(x_n))\uparrow J(s(x))$ by the normality of $J.$ Note that by the construction of $J$ we have  \begin{equation}\label{eq:suppt}
	s(T (x_n) )= s(bJ (x_n) ) 
	=s(b)\wedge s (J(x_n))
	=J(1) \wedge  J(s(x_n))  =J(s(x_n))=0
	\end{equation}  for all $n\geq 1,$ where the second equality follows from the fact that spectral projections of $b$ commute with $J (x_n) $ and the third equality follows from the fact that $J (e)\leq J(1)$ for any projection $e\in\mathcal M$ as $J$ is positive. Thus, we have $J(s(x))=0.$ Since $J$ is injective, this means $s(x)=0.$ Therefore, $x=|\operatorname{Re}\,y| =0.$ Similarly $|\operatorname{Im}\,y| =0$ and hence $y=0$.
	
	(iii) We first prove the surjectivity of $J$. Note that by the surjectivity of $T$, for any $\tau$-finite projection $e$ there exists some $x\in L_p(\mathcal  M)$ with $T(x)=e$. As in the proof of (c)$\Rightarrow$(a) in (ii), it suffices to consider the case where $T(x)=e$ for some positive $x.$ Take a sequence $(x_n)_{n\geq 1}$ as in Lemma \ref{1.10}.  We claim that $s(T(x_n))\uparrow e$. Indeed, since $T$ is positive and $x_n\leq x,$ we have $Tx_n\leq e$ for all $n\geq 1.$ Therefore, $Tx_n$ is bounded for each $n\geq 1.$ Note that $s(Tx_n)\leq e.$ Now \[(e-\vee_{n\geq 1}s(Tx_n))(e-Tx_n)= e-\vee_{n\geq 1}s(Tx_n).\] 
	Therefore, we have  
	\begin{align*}
	\|e-\vee_{n\geq 1}s(Tx_n)\|_p&\leq\|e-\vee_{n\geq 1}s(Tx_n)\|_\infty\|e-Tx_n\|_p \\
	&\leq 2\|e-Tx_n\|_p\to 0, \quad \text{as}\  n\to\infty.	\end{align*}
	This implies that $e-\vee_{n\geq 1}s(Tx_n) =0$.  So we obtain our claim. We have $J(s(x_n))\uparrow J(s(x))$ by the normality of $J$ and $s(Tx_n)=J(s(x_n))$ for all $n\geq 1 $ as in \eqref{eq:suppt}. Thus $J(s(x))=e.$ Since the span of $\tau$-finite projections is $\operatorname{w}^*$-dense in $\mathcal M,$ we see that $J(\mathcal M)$ is $\operatorname{w}^*$-dense in $\mathcal M$. Thus $J(\mathcal M)=\mathcal M.$
	
	Clearly, we have that $J(1)\leq 1.$ Therefore, by surjectivity there exists  $x\in \mathcal M$ such that $J(x)=1-J(1).$ Then  $J(x)=J(1)J(x)=J(1)(1-J(1))=0.$ Thus $s(b)=J(1)=1.$

	Now we prove that $T$ is injective on $L_p(p_1\mathcal  M p_1)$. First, note that the operator $T\big{|}_{L_p(p_1\mathcal M p_1)}$ also separates supports and has the representation $p_1xp_1\mapsto J(p_1)bJ(p_1)J(p_1xp_1).$ Therefore, by (ii), it is enough to show that the map $p_1xp_1\to J(p_1xp_1)$ is injective. Now if $J(p_1xp_1)=0$ for some positive $x$, then by \eqref{repp}, $\tau(\rho_T p_1xp_1)=0.$ Recall that $p_1=s(\rho_T)$. By the faithfulness of $\tau$ we obtain that $(\rho_T)^{1/2} x (\rho_T)^{1/2}=0 $ and $p_1xp_1=0.$ Note that the equality $\|T(x)\|_p^p=\tau(\rho_T |x|^p)$ can be extended by density to all $x\in L_p (\mathcal M)$ by \eqref{repp}. So by a similar argument we see that $T\big{|}_{L_p(p_1\mathcal  M p_1)}$ is also injective and $\ker T = L_p(p_0\mathcal  M p_0)$.
	
	Since $T\big{|}_{L_p(p_1\mathcal  M p_1)}$ is bounded,   so is $T\big{|}_{L_p(p_1\mathcal  M p_1)}^{-1}$ by the open mapping theorem. So we may find some constant $C>0 $ such that for all $x\in \mathcal{S}(\mathcal{M})_{+},$ \[\|T(p_1xp_1)\|_p\geq C\|p_1xp_1\|_p.\] This implies that $\tau(\rho_T p_1xp_1)\geq C\tau(p_1xp_1)$ for all $x\in \mathcal{S}(\mathcal{M})_{+}.$ In particular  $p_1\rho_T \geq Cp_1 $, as desired. 
\end{proof}

The following lemma is elementary. We include here for the convenience of the reader.
\begin{lem}\label{proless}
	Let $p,q\in   \mathcal M  $ be two projections with $pqp=p.$ Then we have $p\leq q.$
\end{lem}
\begin{proof}We write the decomposition
	\[q=x+y+ y^* +z,\quad x= pqp, \quad y= pq(1-p), \quad
	z= (1-p)q(1-p).\] By our assumption $x=p.$ Note that $q$ is a projection. Hence
	\[x=pqp=pq^2p=p(x+y+ y^* +z)^2p = x+ y y^*.\]
	Thus $y=0 $ and $q-p=z\geq 0$. 
\end{proof}

 To this end we need the following proposition. For technical simplicity, we will only consider the case of finite von Neumann algebras, where the operator $b$ becomes measurable and is in $L^1.$
\begin{prop}\label{duss}
	Let $1<p<\infty $ and $1/p +1/p' =1$. Assume that $\mathcal M$ is a finite von Neumann algebra equipped with a normal faithful tracial state $\tau$ and that  $T:L_p(\mathcal  M)\to L_p(\mathcal  M)$ is a positive Lamperti operator. If the adjoint operator $T^*:L_{p'}(\mathcal  M)\to L_{p'}(\mathcal  M)$ is also Lamperti, then $J(\mathcal  M)= J(1)\mathcal  M J(1)$.
\end{prop}
\begin{proof} Assume by contradiction $J(\mathcal  M)\neq J(1)\mathcal  M J(1).$ Then there exists a nonzero projection $f_1\in J(1)\mathcal  M J(1)\setminus J(\mathcal  M) $ (if not, then $J(\mathcal  M)$ contains the span of all projections in $ J(1)\mathcal  M J(1) $ which is a $\operatorname{w}^*$-dense subspace).  Let us define \[e_1=\wedge\{J(e) :f_1\leq J(e)\leq J(1), e\in \mathcal P (\mathcal M)\}.\] Then  $f_1\leq e_1$. Recall   that $J$ is a normal Jordan $*$-homomorphism. According to Lemma \ref{decom}, we may write $J$ as a direct sum $J=J_1 + J_2$, where $J_1$ is a normal $*$-homomorphism and $J_2$ is a normal $*$-anti-homomorphism.  Then for a finite family of projections $q_1,\ldots,q_n$, we have
	\[J_1(\vee_{1\leq i\leq n}q_i^\bot)
	=J_1(s(\sum_{i=1}^{n}q_i^\bot))
	=s(J_1(\sum_{i=1}^{n}q_i^\bot)) =\vee_{1\leq i\leq n}J_1(q_i^\bot),\] 
	whence $J_1(\wedge_{1\leq i\leq n}q_i  ) = \wedge_{1\leq i\leq n} J_1 (q_i)$. Similarly $J_2(\wedge_{1\leq i\leq n}q_i  ) = \wedge_{1\leq i\leq n} J_2 (q_i)$. Hence we have 
	\[J (\wedge_{1\leq i\leq n}q_i  ) = \wedge_{1\leq i\leq n} J  (q_i).\]
	By the $\operatorname{w}^*$-closeness of $J(\mathcal M)$, we see that there exists a projection   $\widetilde{e}_1 \in \mathcal M$ with $e_1 = J(\widetilde{e}_1 )$.  Denote $f_2=e_1-f_1.$ Clearly, $f_2$ is a projection in $J(1)\mathcal  M J(1)\setminus J(\mathcal M).$ Now, choose $e_2$ and $\widetilde{e}_2$ similarly as before corresponding to $f_2$. Note that we have $0\neq e_1-f_1= f_2\leq e_2.$ Therefore, we have $e_1\wedge e_2\neq 0$. Thus, $e_1e_2\neq 0.$ Note that by construction, 
	\begin{equation}\label{eq:seperate}
	f_1f_2=f_2f_1=0.
	\end{equation}  Since $T$ is positive, so is $T^*$. Note that $\tau$ is finite and hence all projections are $\tau$-finite. Thus for $i=1,2,$  $T^*(f_i)$ is well-defined and $T^*(f_i)\geq 0.$   

	Denote $\overline{e_i}=s(T^* (f_i))$ for $=1,2.$ 
	
	We claim that $J(\overline{e_i})=e_i$ for $i=1,2.$ To establish our claim,
	we first observe   that 
	\begin{equation*}\label{lastkic}\tau(T^*(f_i)\widetilde{e_i})=\tau(f_ibe_i)=\tau(e_i f_ib)=\tau(f_ib) =\tau(f_ibJ(1))=\tau(T^*(f_i)),\end{equation*} 
	and therefore  \[\tau(T^*(f_i)-T^*(f_i)^{\frac{1}{2}}\widetilde{e_i}T^*(f_i)^{\frac{1}{2}})=0,
	\quad T^*(f_i)=T^*(f_i)^{\frac{1}{2}}\widetilde{e_i}T^*(f_i)^{\frac{1}{2}},\quad i=1,2.\]   
	By using the functional calculus for $t\mapsto \chi_{\sigma(T^*(f_i))} (t) t^{-1/2}$, we see that
	\[\overline{e_i}=\overline{e_i}\widetilde{e_i}\overline{e_i}\] for $i=1,2.$ Therefore, by Lemma \ref{proless} we have $\overline{e_i}\leq \widetilde{e_i}$ for $i=1,2.$ Hence, we obtain 
	\begin{equation}\label{dont1}
	J(\overline{e_i})\leq e_i
	\end{equation} for $i=1,2.$
	Note that we have 
	\begin{equation*}\label{dont2}
	0=\tau(T^*(f_i)\overline{e_i}^{\perp})=\tau(f_iT(\overline{e_i}^{\perp}))=\tau(f_ibJ(\overline{e_i}^{\perp})).
	\end{equation*}
	Together with the fact that $b$ commutes with the projection $J(\overline{e_i}^{\perp})$, we get \begin{equation*}\label{dont3}b^{\frac{1}{2}}J(\overline{e_i}^{\perp})f_iJ(\overline{e_i}^{\perp})b^{\frac{1}{2}}=0.\end{equation*} 
	Therefore  \begin{equation*}\label{dont4}s(b)J(\overline{e_i}^{\perp})f_iJ(\overline{e_i}^{\perp})s(b)
	=J(1)J(\overline{e_i}^{\perp})f_iJ(\overline{e_i}^{\perp})J(1)=J(\overline{e_i}^{\perp})f_iJ(\overline{e_i}^{\perp})=0.\end{equation*} Thus 
	\begin{equation*}\label{dont5}
	0=\tau(J(\overline{e_i}^{\perp})f_iJ(\overline{e_i}^{\perp}))=\tau(f_iJ(\overline{e_i}^{\perp}))=\tau(f_iJ(\overline{e_i}^{\perp})f_i).
	\end{equation*}
	Therefore    $f_iJ(\overline{e_i}^{\perp})f_i=0$ for $i=1,2.$ Note that $f_i\leq J(1)$. So we have \[f_i = f_i J(1)f_i =f_iJ(\overline{e_i})f_i\] for $i=1,2.$ Hence by Lemma \ref{proless} we have $f_i\leq J(\overline{e_i})$ for $i=1,2.$ From this, using \eqref{dont1} and the minimality of $e_i$ we conclude that $J(\overline{e_i})=e_i$ for $i=1,2.$
	
	Now we obtain
	\[J_1(s(T^*(f_1)) s(T^*(f_2))) + J_2 (s(T^*(f_2)) s(T^*(f_1)))  =
	J(s(T^*(f_1)))J(s(T^*(f_2)))=e_1 e_2\neq 0\] by the above claim. This yields that $s(T^*(f_1))s(T^*(f_2))\neq 0$, and in particular we have $T^*(f_1)T^*(f_2)\neq 0$. However we have $f_1f_2=0$ by \eqref{eq:seperate}. So $T^*$ is not Lamperti, which leads to a contradiction. 
\end{proof}

 We need the following lemma which was proved in \cite{kan79thesis} in the classical case. The proof of our lemma is completely different from \cite{kan79thesis} but again restricted to finite von Neumann algebras only.

\begin{lem}\label{extension}
	Let $\mathcal M$ be a finite von Neumann algebra and $\tau$ be a normal faithful tracial state on $\mathcal M$. Let $1\leq p<\infty.$ Let $T:L_p(\mathcal M)\to L_p(\mathcal M)$ be a  positive Lamperti operator with the representation $T(x)=bJ(x)$ for all $x\in\mathcal M.$ Then $J$ and $T$ can be extended continuously  to maps on $ L_0(\mathcal M) $ with respect to the topology of convergence of measure. Moreover,   $Tx=bJ(x)$ for all $x\in L_0(\mathcal M).$
\end{lem}
\begin{proof}First we show that $J:\mathcal M\to \mathcal M$ is continuous in the topology of convergence of measure on $L_0(\mathcal M)$. Take a sequence $(x_n)_{n\geq 1}\subseteq\mathcal M_+$ which converges to $0$ in measure, that is,  $\tau(e_\varepsilon^{\perp}(x_n))\to 0$ as $n\to\infty$ for all $\varepsilon>0.$ For any $x\in\mathcal M_+,$ the restriction of $J$ on the abelian von Neumann subalgebra generated by $x$ is a classical normal $*$-homomorphism. Note that $J(x)\geq \varepsilon$ iff $J(x)=J(1)J(x)J(1)\geq \varepsilon J(1)$. It follows that $J(e_\varepsilon^{\perp}(x))=e_\varepsilon^{\perp}(J(x))$ for all $\varepsilon>0.$ 
	
	We also have $\tau(b^pJ(e_\varepsilon^{\perp}(x_n)))\leq C\tau(e_\varepsilon^{\perp}(x_n)).$ This shows that \begin{equation}\label{avt1}\lim\limits_{n\to\infty}\tau(b^pJ(e_\varepsilon^{\perp}(x_n)))=0.\end{equation} Let $f_k$ denote the spectral projection $\chi_{[2^k,2^{k+1})}(b^p),$ $k\in\mathbb Z.$ Note that we have \begin{equation}\label{bhotehobe}b^pJ(e_\varepsilon^{\perp}(x_n))\geq 2^kJ(e_\varepsilon^{\perp}(x_n))f_k.\end{equation} Therefore, by \eqref{avt1} we have \begin{equation}\label{avt2}\tau(J(e_\varepsilon^{\perp}(x_n))f_k)\to 0\end{equation} as $n\to\infty$ for all $k\in\mathbb Z.$ Note that since $J(e_\varepsilon^{\perp}(x_n))$ is a projection and contained in $s(b^p),$ we have \begin{equation}\label{avtt4}J(e_\varepsilon^{\perp}(x_n))=\sum_kJ(e_\varepsilon^{\perp}(x_n))f_k.\end{equation}  Let us fix $\delta>0.$ Note that $\sum_{k}f_k\leq J(1)$ and $\tau$ is finite so that $\sum_{k}\tau(f_k ) <\infty $. Using \eqref{avt2} we choose $n$  large enough so that $\tau(J(e_\varepsilon^{\perp}(x_n))f_k)\leq \frac{\delta}{2s}$ for $|k|\leq s$ and $\sum_{|k|>s}\tau(f_k)<\delta.$ Then by \eqref{avtt4} and \eqref{bhotehobe}  we have \begin{equation}\label{amig}\tau(J(e_\varepsilon^{\perp}(x_n))=\sum\limits_{|k|\leq s}\tau(J(e_\varepsilon^{\perp}(x_n))f_k)+\sum\limits_{|k|> s}\tau(J(e_\varepsilon^{\perp}(x_n))f_k)\leq \delta+\sum_{|k|>s}\tau(J(e_\varepsilon^{\perp}(x_n))f_k).\end{equation} 
	Also note that \begin{equation}\label{amigg}\sum_{|k|>s}\tau(J(e_\varepsilon^{\perp}(x_n))f_k)\leq \sum_{|k|>s}\|J(e_\varepsilon^{\perp}(x_n)\|_\infty\tau(f_k)\leq \sum_{|k|>s}\tau(f_k)<\delta
	\end{equation}as $J$ is a contraction. Together with  \eqref{amig} and \eqref{amigg} this establishes that $\lim\limits_{n\to\infty}\tau(J(e_\varepsilon^{\perp}(x_n))=0.$ Therefore, $J$ is continuous in the topology of measure. Since $\mathcal M$ is dense in $L_0(\mathcal M)$, we can extend uniquely $J$ to a map on $L_0(\mathcal M)$, which is also continuous. Now we may extend $T$ to a linear map on $L_0(\mathcal M)$ by setting $ {T}x=b\tilde{J}(x).$ This completes the proof of the lemma. 
\end{proof}	
Kan \cite{kan78erglamperti} showed that the converse of Proposition \ref{duss} is also true in the classical setting. Though we could not establish the analogue for the noncommutative setting, we may prove a partial result.
\begin{prop}
	Let $1<p<\infty$ and $\mathcal M$ be a finite von Neumann algebra. Let $T:L_p(\mathcal M)\to L_p(\mathcal M)$ be a  positive and surjective Lamperti operator. Then $T^*$ is again  Lamperti.
\end{prop}
\begin{proof}   Since $T$ is onto, it follows from Proposition \ref{properties} that $J$ is unital and onto, and moreover the restriction $J:p_1\mathcal Mp_1\to\mathcal M$ is a normal  Jordan $*$-isomorphism. Consider a normal faihtful tracial state $\tau$ on $\mathcal M$; together with Lemma \ref{decom}, we note that ${\varphi}\coloneqq \tau\circ J$ is a normal tracial state on $\mathcal M$. Thus we may write $\varphi = \tau (t \cdot )$ for some positive element $t\in L_1 (\mathcal M,\tau)$ which commutes with $\mathcal M$.   By Lemma \ref{extension}, the elements $\widetilde{b}=J|_{L_0(p_1\mathcal Mp_1)}^{-1}(b)t$ and $S(y)=\widetilde{b}J|_{L_0(p_1\mathcal Mp_1)}^{-1}(y)$ can be well-defined for $y\in\mathcal M.$ We claim that the adjoint operator of $T:L_p(\mathcal M,\tau)\to L_p(\mathcal M,\tau)$ is $S.$ Indeed, note that
	\begin{align*}
	\begin{split}
	\tau(xS(y))&=\tau(x\widetilde{b}J|_{L_0(p_1\mathcal Mp_1)}^{-1}(y))\\
	&=\tau(xJ|_{L_0(p_1\mathcal Mp_1)}^{-1}(b)tJ|_{L_0(p_1\mathcal Mp_1)}(y))\\
	&=\tau(txJ|_{L_0(p_1\mathcal Mp_1)}^{-1}(b)J|_{L_0(p_1\mathcal Mp_1)}(y))\\
	&=\varphi(xJ|_{L_0(p_1\mathcal Mp_1)}^{-1}(b)J|_{L_0(p_1\mathcal Mp_1)}(y))\\
	&=\tau(J(xJ|_{L_0(p_1\mathcal Mp_1)}^{-1}(b)J|_{L_0(p_1\mathcal Mp_1)}^{-1}(y)))\\
	&=\tau(J(x)by)=\tau(T(x)y)\\
	\end{split}
	\end{align*} for all $x\in\mathcal M,y\in\mathcal M.$ This establishes the claim. Clearly, $S$ is a Lamperti operator by Theorem \ref{charac}. This completes the proof.
\end{proof}	

We are ready to prove the following key description of doubly Lamperti operators on noncommutative $L_p$ spaces. 

\begin{thm}\label{CHARACDL}Let $\mathcal M $ be a finite von Neumann algebra. Let $1<p<\infty.$ Suppose that $T:L_p(\mathcal M)\to L_p(\mathcal M)$ is a positive Lamperti operator with the representation  $Tx=bJ(x)$ as in Theorem \ref{charac}.   Then there exist an element $\theta\in\mathcal M$ and a positive Lamperti contraction $S:L_p(\mathcal M)\to L_p(\mathcal M)$  such that  $T^n=\theta_nS^n,$ where 
	\begin{enumerate}
		\item $S$ is a positive  Lamperti contraction which vanishes on $L_p(p_0\mathcal M p_0)$ and is isometric on $L_p(p_1\mathcal M p_1);$
		\item $\theta_n$ is a positive element in $\mathcal M$ of the form $\theta_n=\theta J(\theta)\cdots J^{n-1}(\theta)$ and $\theta_nS^n(x)=S^n(x)\theta_n$ for all $n\geq 1$ and $x\in \mathcal M;$
		\item for all $n\geq 1,$ $\|T^n\|_{L_p(\mathcal M)\to L_p(\mathcal M)}\leq\|\theta_n\|_\infty.$ Moreover, the equality holds if  the adjoint operator $T^*: L_{p'} (\mathcal M)\to L_{p'} (\mathcal M) $ for $1/p +1/p' =1$ is also Lamperti.
	\end{enumerate}
\end{thm}
\begin{proof} Without loss of generality we assume $\|T\|_{L_p(\mathcal M)\to L_p(\mathcal M)}\leq 1.$ The general case follows by considering the contraction $T/\|T\|$ in the proof. 
	
	(i) Recall that $p_0,p_1\in\mathcal{Z}(\mathcal M)$, $p_0+p_1=1$, and $  \rho_T =p_1\rho_T p_1 $.  Note that we may see from the proof of (ii) in Proposition \ref{properties} that $T$ and $J$ are injective on $L_p(p_1 \mathcal M p_1)$ and $p_1 \mathcal M p_1$ respectively.
	Clearly, $(p_1\rho_T p_1)^{-1} $ is well-defined as a densely defined operator in $L_0(p_1\mathcal M p_1)_{+}.$ 
	We use Lemma \ref{extension} and define \[\widetilde{\rho_T }= J \left( (p_1\rho_T p_1 )^{-\frac{1}{p}}\right), \quad \widetilde{b}= b\widetilde{\rho_T }.\] 
	Then the spectral projections of $\widetilde{b}$ commute with $J(\mathcal M)$ since the operators $p_1$ and $\rho_T$ belong  to the center of $\mathcal M$. Also, we observe that \[s(\widetilde{b})=s(b)\wedge s(\widetilde{\rho_T }) = J(1) \wedge J(s((p_1\rho_T p_1 )^{-\frac{1}{p}}) )=J(p_1)=J(1)\] as we have $J(p_0)=0$, according to the fact $T(p_0)=0$ in Proposition \ref{properties}(i). Define   the positive linear operator
	\[S(x)= \widetilde{b}J(x),\quad x\in \mathcal M.\] By Theorem \ref{charac} and Remark \ref{kotokikorarchilobaki}, $S$ is a Lamperti operator. 
	
	Applying \eqref{repp} to $S$, we have
	\[\tau(\rho_S p_0xp_0)=\tau(\widetilde{b}^pJ(p_0xp_0))
	=\tau\left( {b}^pJ\left((p_1\rho_T p_1 )^{-1}p_0xp_0\right)\right)=0 \] for all 	   $x\in\mathcal M_{+}, $ which means that  $p_0\rho_S p_0=0.$ Similarly, for all $x\in\mathcal M_{+}$ we have
	\begin{align*}
	\tau(\rho_S p_1xp_1)
	&=\tau(\widetilde{b}^pJ(p_1xp_1)) =\tau\left( {b}^pJ\left((p_1\rho_T p_1 )^{-1}p_1xp_1\right)\right) \\
	&=\tau(\rho_T (p_1\rho_T p_1 )^{-1}p_1xp_1) = \tau(p_1xp_1). 
	\end{align*} This shows that $p_1\rho_S p_1=p_1.$ Applying \eqref{repp} to $S$ again, we see that  $S|_{L_p(p_1\mathcal M p_1)}$ is an isometry and $S|_{L_p(p_0\mathcal M p_0)}=0$. This completes the proof for (i).
	
	(ii) Define $\theta=J(\rho_T )^{\frac{1}{p}}$ and $\theta_n=\theta J(\theta)\cdots J^{n-1}(\theta).$ Recall that $\rho_T$ is in the center of $\mathcal M$, so $\rho_T$ commutes with $J^k(\theta)$ for all $k\geq 0$, and applying the Jordan homomorphism $J$ we see that  $\theta$ commutes with all $J^k(\theta)$. In particular $\{\rho_T,\theta,J(\theta)\}$ is a commuting family. We see easily by induction that $(J^k(\theta))_{k\geq 0}$ is a commuting family. In particular $\theta_n\geq 0$. Note that $\theta S(x)=S(x)\theta$ for all $x\in L_0(\mathcal M).$ We claim that $T^n=\theta_n S^n,$ for all $n\geq 1$. Indeed, for $n=1,$ recalling that we have observed $J(1)=J(p_1)$ in {(i)}, we see that \[\theta S(p_1 xp_1)  
	=J(\rho_T )^{\frac{1}{p}} b \widetilde{\rho_T } J(p_1 xp_1) 
	= \textbf{}
	J(\rho_T )^{\frac{1}{p}}J \left( (p_1\rho_T p_1 )^{-\frac{1}{p}}p_1xp_1\right)  =bJ(1)J(x)=T(x).\]
	Assume by induction that  $T^n=\theta_n S^n$. Then
	\begin{align*}
	T^{n+1}(x) &= T (\theta_n S^n (x))
	= bJ(\theta_{n}) J(S^n(x))= b  J(\theta_n)  \widetilde{b}^{-1} S^{n+1} (x) \\
	&= J(\theta_n) \widetilde{\rho_T}^{-1}  S^{n+1} (x) = J(\theta_n)  J  (   \rho_T  ^{ \frac{1}{p}} )   S^{n+1} (x) = \theta  J(\theta_n)  S^{n+1} (x) = \theta_{n+1}  S^{n+1} (x).
	\end{align*}
	
	(iii)
	It is obvious that $\|T^n\|_{L^p(\mathcal M)\to L^p(\mathcal M)}\leq\|\theta_n\|_{\infty}.$ 
	Assume that $T^*$ is Lamperti. Then by Proposition \ref{duss}, $J (\mathcal M)=J(1)\mathcal M J(1) $ and we see inductively $J^n (\mathcal M)=J^n(1)\mathcal M J^n(1).$ On the other hand, we have proved in {(ii)} that $(J^k(\theta))_{k\geq 0}$ is a commuting family, so we obtain $\theta_n \in J^n(1)\mathcal M J^n(1)$, whence $\theta_n\in J^n (\mathcal M)$. 
	
	Recall moreover that  $J(p_0)=0$. Thus we may write $\theta_n=J^n ( x_n )$ for some $x_n\in p_1 \mathcal M_{+} p_1$. Let $\|\theta_n\|_\infty>A $ and take a spectral projection $q=e_A^\bot (\theta_n) \in J^n( \mathcal M )$ so that $q\theta_n\geq A\theta_n$. Note that $J(x)\geq \varepsilon$ iff $J(x)=J(1)J(x)J(1)\geq \varepsilon J(1)$. It follows that $J(e_\varepsilon^{\perp}(x))=e_\varepsilon^{\perp}(J(x))$ for all $\varepsilon>0.$ So we may write $q=e_A^\bot (\theta_n)=J^n (e_A^\bot (x_n))$. Denote $e=e_A^\bot (x_n)$. Note that \[ J^n(e) S^n (e)  
	=J^n(e) J^n(e) \widetilde{b}J(\widetilde{b}) \cdots J^n(\widetilde{b})  = S^n (e).\] Therefore, using $T^n(e)=\theta_n S^n(e) = S^n(e) \theta_n  $    we obtain that 
	\[\|T^n(e)\|_p =\|\theta_n  J^n(e) S^n(e)\|_p\geq A\|S^n(e)\|_p.\] 
	This implies that $\|T^n\|_{L_p(\mathcal M)\to L_p(\mathcal M)}\geq A$ as $S$ is an isometry on $L_p(p_1\mathcal M p_1).$  This completes the proof of the theorem.		
\end{proof}

Based on Theorem \ref{main} and the above result,   we   conclude the proof of  the main result.
\begin{proof}[Proof of Theorem \ref{main1}]
	By Theorem \ref{CHARACDL}, there is a positive   Lamperti contraction $S$ such that for all $x\in \mathcal M_{+}$ and $n\in\mathbb N$, we have
	\[T^n(x) =\theta_n S^n (x) \leq \|\theta_n\|_\infty S^n (x) = \|T^n\| S^n (x) \leq K S^n(x).\]
	Hence  \[  \frac{1}{ n+1}\sum_{k=0}^nT^kx \leq K  \frac{1}{ n+1}\sum_{k=0}^n S^kx .  \]  The proof is complete  according to Theorem \ref{main}.
\end{proof}	

\section{Ergodic theorems beyond Lamperti operators} \label{exa}As pointed out previously, Theorem \ref{main} and Theorem \ref{main1} apply to quite general classes of positive operators on classical $L_p$-spaces. However in the noncommutative setting, we may explore other novel and natural examples beside these categories, showing sharp contrast to the classical setting. In this section, we will illustrate two ergodic theorems outside the scope of Theorem \ref{main} or Theorem \ref{main1}.
\subsection{Positive invertible operators which are not Lamperti}\label{sect:kan exa}
In the classical setting we have the following examples of Lamperti operators.
\begin{prop}[\cite{kan78erglamperti}] \label{prop:kan exa}
	\emph{(i)}  Let $1<p<\infty.$ Let $\Omega$ be a $\sigma$-finite measure space. Let $T:L_p(\Omega)\to L_p(\Omega)$ be a bounded positive operator with positive inverse. Then $T$ is Lamperti.
	
	\emph{(ii)} Let $T$ be an invertible nonnegative $n\times n$ matrix such that the set $\{T^k:k\in\mathbb Z\}$ is uniformly bounded in any equivalent matrix norm. Then $T$ is periodic and Lamperti.
\end{prop} 
We provide the following example  which illustrates that there is no reasonable analogue of Kan's above examples for the noncommutative setting. 
\begin{exa}\label{ex:exa}
	Let $1\leq p<\infty $ and $r$ be an invertible matrix $2\times 2$ matrix.
	Define 
	\[T:S_p^2\to S_p^2,\quad T(x)=rxr^*.\]
	Clearly, $T$ is completely positive map, and so is the inverse map $T^{-1}(x)=r^{-1} x (r^{-1})^*$. Note that \[e=\left(
	\begin{array}{ccccc}
	1 & 0\\
	0 &  0\\
	\end{array}
	\right) \quad \text{and} \quad
	  f=\left(
	\begin{array}{ccccc}
	0 & 0\\
	0 &  1\\
	\end{array}
	\right)\] are two orthogonal projections with $ef=fe=0$. But if we take \[r=\left(
	\begin{array}{ccccc}
	1 & 1\\
	\alpha &  \beta\\
	\end{array}\right)\] with $\alpha,\beta\in\mathbb R $ and $1+\alpha\beta\neq 0,$ it is easy to see that $T(e)T(f)\neq 0.$ So $T$ is not Lamperti.
	
	Moreover, consider $\alpha=0,\beta=-1$. Then $r^{-1}=r$ and  $r^2=1_{M_2}$. So  \[\sup_{n\in\mathbb Z}\|T^n\|_{cb, S_\infty^2\to  S_\infty^2} \leq \sup_{n\in\mathbb Z} \|r^n\|_{\infty}^2<\infty.\]
	Since the operator space of linear operators on $M_2$ is finite dimensional, so $(T^k)$  is uniformly bounded with respect to any equivalent operator norm. 
	So we obtain an analogue of operators satisfying {(i)} and {(ii) }of Proposition \ref{prop:kan exa} for the noncommutative setting, but they are not Lamperti. Moreover, we can observe that $\|T(f)\|_p=2.$ Therefore, $T$ is not a contraction for all $1\leq p\leq\infty.$
\end{exa}

Denote $K=\sup_{n\in\mathbb Z}\|T^n\|_{ S_p^2\to S_p^2}$. The above discussions also mean that Theorem \ref{main} is not applicable to obtain the crucial constant $KC_p$ for the maximal ergodic inequality associated with $T$   since $T$ is not a contraction on $S_p^2$.  Moreover Theorem \ref{main1} is not applicable neither since $T$ is not Lamperti.
However, this example  still satisfies the maximal ergodic inequalities with crucial constant $KC_p$ according to the following result in \cite{hongliaowang15erg}. The crucial constant $KC_p$ is not stated explicitly in \cite{hongliaowang15erg} but is implicitly contained in the proof.
\begin{thm}[\cite{hongliaowang15erg}]\label{inve} Let $1<p<\infty.$ Let $\mathcal M$ be a von Neumann algebra with a normal semifinite faithful trace. Suppose $T:L_p(\mathcal M)\to L_p(\mathcal M)$ be a bounded invertible positive operator with positive inverse, such that $\sup_{n\in\mathbb Z}\|T^n\|_{L_p(\mathcal M)\to L_p(\mathcal M)}=K<\infty.$ Then 
	\[\Big\|\Big(\frac{1}{2n+1}\sum_{k=-n}^nT^kx\Big)_{n\geq 0}\Big\|_{L_{p}({\mathcal{M}};\ell_{\infty})}\leq KC_{p }\|x\|_p\] for all $x\in L_p(\mathcal M).$				
\end{thm}
Note that $S_p^2$ and $S_\infty^2$ are isomorphic as finite dimensional Banach spaces, so the positive invertible operator $T$ given in Example \ref{ex:exa} with positive inverse associated with $\alpha=0,\beta=-1$ satisfies
\[K_p\coloneqq\sup_{n\in\mathbb Z}\|T^n\|_{ S_p^2\to \mathbb S_p^2}   <\infty.\]
Applying the above theorem, we have 
\[\Big\|\Big(\frac{1}{2n+1}\sum_{k=-n}^n T^kx\Big)_{n\geq 0}\Big\|_{L_{p}({\mathcal{M}};\ell_{\infty})}\leq K_pC_{p }\|x\|_p,\quad x\in S_p^2.\]

\subsection{Junge-Le Merdy's example}In this subsection, we take Junge-Le Merdy's examples \cite{jungelemerdy07dilation} and establish the noncommutative ergodic theorem for them. That is, we prove Proposition \ref{atm}.

\begin{proof}[Proof of Proposition \ref{atm}]Let $(e_{ij})_{i,j=1}^k$ be the standard basis of $S_p^k.$ Following the examples in \cite[Section 5]{jungelemerdy07dilation}, we define the operators on $S_p^k$ as 
	\[T_1(x)=\sum_{i=1}^ka_i^*xb_i,\ T_2(x)=\sum_{i=1}^kb_i^*xa_i,T_3(x)=\sum_{i=1}^ka_i^*xa_i,
	\ T_4(x)=\sum_{i=1}^kb_i^*xb_i,\ x\in S_p^k,\] where $a_i=e_{ii}$ and $b_i=k^{-\frac{1}{2p}}e_{1i}$ for $1\leq i\leq k.$ By \cite{jungelemerdy07dilation} each $T_i$ is a  contraction for $1\leq i\leq 4.$ We define
	\[T=\frac{1}{4}(T_1+T_2+T_3+T_4).\] Then $T$ is completely positive and completely contractive. For any positive element $x$, a straightforward calculation yields \[ (T(x))_{ij}=0,\quad \forall\, i\neq j.\] 
	Let $\mathcal D_p^k$ be the diagonal $L_p$-subspace of $S_p^k$. Then $\mathcal D_p^k$ becomes a commutative $\ell_p$-space and $\operatorname{ran} (T)\subseteq \mathcal D_p^k$. In particular, the restriction $ T  {|_{\mathcal D_p^k}}:\mathcal D_p^k \to \mathcal D_p^k$ is a positive contraction on the commutative $\ell_p$ space $\mathcal D_p^k$. Therefore, by Akcoglu's ergodic theorem \cite{akcoglu75erg}, we have 
	\[\Big\|\Big(\frac{1}{ n+2}\sum_{m=0}^n T^{m+1} y\Big)_{n\geq 0}\Big\|_{L_{p}({\mathcal{M}};\ell_{\infty})}\leq C_p\|Ty\|_p\leq C_p\|y\|_p\] for all $y\in S_p^k.$ Therefore, for all $x\geq 0,$ we have from above
	\[\Big\|\Big(\frac{1}{n+1}\sum_{m=0}^n T^{m} x\Big)_{n\geq 0}\Big\|_{L_{p}({\mathcal{M}};\ell_{\infty})}
	\leq\Big\|\Big(\frac{1}{ n+2}\sum_{m=1}^n T^{m+1} x\Big)_{n\geq 0}\Big\|_{L_{p}({\mathcal{M}};\ell_{\infty})} + \|x\|_p 
	\leq (C_p +1)\|x\|_p.\] We can choose $k$ to be large enough so that $T$ does not admit a dilation (see \cite{jungelemerdy07dilation}). This completes the proof.
\end{proof}
\begin{rem}
	Note that  we cannot directly apply Theorem  \ref{main} to the non-dilatable operator $T:S_p^k \to S_p^k$ (see Remark \ref{rem:dilation}). However, the following property is applicable, which can be easily deduced from above arguments together with Theorem  \ref{main}:
	Let $1<p<\infty.$ Let $T:L_p(\mathcal M)\to L_p(\mathcal M)$ be a positive contraction such that for some positive integer $k,$ we have $\operatorname{ran}\,(T^k)\subseteq L_p(\mathcal N)$ where $\mathcal N\subseteq\mathcal M$ is a von Neumann subalgebra and $T\big|_{L_p(\mathcal{N})}\in \overline{\operatorname{conv}}^{sot}(\mathbb{SS}^+(L_p(\mathcal{N})))$, then $T$ admits a maximal ergodic inequality as above. 
\end{rem}

\subsection*{Acknowledgment}
The authors would like to thank Professor Quanhua Xu for helpful comments, and also thank C\'edric Arhancet for his interesting communications on Theorem \ref{thm:cedric} and Proposition \ref{prop:cp}. Guixiang Hong was partially supported by NSF of China (No. 12071355) and the Fundamental Research Funds for the Central Universities. Samya Kumar Ray acknowledges DST-INSPIRE Faculty Fellowship No. DST/INSPIRE/04/2020/001132. Simeng Wang was partially supported by a public grant as part of the FMJH, the ANR Project (No. ANR-19-CE40-0002), the
Fundamental Research Funds for the Central Universities No. FRFCUAUGA5710012222 and the NSF of China (No.12031004). We sincerely thank the anonymous referee for several suggestions which considerably improved the presentation of this paper.

\end{document}